\documentclass[twoside,11pt]{article}

\usepackage[english,american]{babel}

\usepackage{url}
\usepackage{epsfig} 
\usepackage{times} 
\usepackage{amsmath}
\usepackage{amsfonts}
\usepackage{amssymb}  
\usepackage{mathtools}

\usepackage{xcolor}
\usepackage{transparent}
\graphicspath{{media/}}
\usepackage{tikz}

\usepackage{epstopdf}
\usepackage{units}
\usepackage{pgfplots}

\usepackage{arydshln}
\usepackage{perpage}
\MakePerPage{footnote}

\newtheorem{assumption}{Assumption}

\pgfplotsset{every tick label/.append style={font=\footnotesize}}

\usetikzlibrary{arrows}
\usetikzlibrary{shapes.geometric}
\usetikzlibrary{arrows.meta}

\usetikzlibrary{external}
\newcommand{\TT}{\ensuremath{\mathsf{\tiny{T}}}}
\newcommand{\T}{^{\TT}}

\newcommand{\diff}[1][]{\mathrm{d}#1}
\newcommand{\dt}{\diff t }

\newcommand{\hchange}[1]{#1}
\newcommand{\hchangeII}[1]{#1}

\usepackage{lastpage}
\usepackage[preprint]{jmlr2e}

\jmlrheading{22}{2021}{1-\pageref{LastPage}}{3/20; Revised 1/21}{2/21}{20-207}{Michael Muehlebach and Michael I. Jordan}

\ShortHeadings{Optimization with Momentum}{Muehlebach and Jordan}
\firstpageno{1}

\begin{document}

\title{Optimization with Momentum: Dynamical, Control-Theoretic, and Symplectic Perspectives}

\author{\name Michael Muehlebach \email michaelm@berkeley.edu \\
       \addr Department of Electrical Engineering and Computer Sciences\\
       Department of Statistics\\
       University of California\\
       Berkeley, CA 94720-1776, USA
       \AND
       \name Michael I.\ Jordan \email jordan@cs.berkeley.edu \\
       \addr Department of Electrical Engineering and Computer Sciences\\
       Department of Statistics\\
       University of California\\
       Berkeley, CA 94720-1776, USA}

\editor{Ohad Shamir}

\maketitle

\begin{abstract}
We analyze the convergence rate of various momentum-based optimization algorithms from a dynamical systems point of view. Our analysis exploits fundamental topological properties, such as the continuous dependence of iterates on their initial conditions, to provide a simple characterization of convergence rates. In many cases, closed-form expressions are obtained that relate algorithm parameters to the convergence rate. The analysis encompasses discrete time and continuous time, as well as time-invariant and time-variant formulations, and is not limited to a convex or Euclidean setting. In addition, the article rigorously establishes why symplectic discretization schemes are important for momentum-based optimization algorithms, and provides a characterization of algorithms that exhibit accelerated convergence.
\end{abstract}


\section{Introduction}
Optimization problems lie at the heart of many machine-learning formulations. As a result, a better understanding of optimization algorithms, combined with implementations that target distinctive properties of machine-learning problems, have contributed significantly to the recent progress in the field.

One of the most popular methods for large-scale optimization is the (stochastic) gradient method, due to its simplicity, wide applicability, and efficiency. However, in the deterministic setting, where gradients are evaluated exactly, it has been shown that better convergence rates can often be achieved by leveraging two successive gradients~\citep{NesterovOrgPaper}. Based on analogies to mechanical systems these methods are referred to \emph{momentum-based optimization algorithms}, and can be viewed as particular discretizations of continuous-time harmonic oscillators, \citep{PolyakHeavyBall}. Yet, even for the class of strongly convex functions, most proofs that establish the superior convergence are algebraic and provide little qualitative understanding. As a consequence, there is little guide to the generality or robustness of the acceleration phenomenon across instances of optimization problems.

This article characterizes the convergence rate of momentum-based optimization algorithms by taking fundamental topological properties into account. In many cases our analysis leads to closed-form expressions that relate the convergence rate to the different algorithm parameters, such as the step size and the damping.  The analysis provides insight into the design of algorithms that, for example, require little tuning when applied to large-scale and ill-conditioned optimization problems. For simplicity of notation, we focus on the Euclidean setting, but we note that the scope of our analysis is not limited to Euclidean problems.  

We will derive convergence rates of the form 
\begin{equation}
|x(t)-x^*|\leq C_\text{c} |x(0)-x^*| \rho_\text{c}(t), \label{eq:conv1}
\end{equation}
where $|x(t)-x^*|$ is a distance measure between the iterate $x(t)$ and the isolated local minimum $x^*$, $t$ refers to time or the iteration number, $C_\text{c}$ is a constant which does not depend on $x(0)$, and $\rho_\text{c}(t)$ is monotonically decreasing, satisfies $\rho_\text{c}(0)=1$, and characterizes the convergence rate. 
An important aspect of the analysis is to characterize how the convergence rate $\rho_\text{c}$ is affected by the shape of the objective function about an isolated local minimum. The local shape is summarized by the (condition) number $\kappa$, which is defined as the ratio between the maximum and minimum curvature about a local minimum. We call an isolated minimum degenerate if the minimum curvature vanishes in at least one direction. The main results and insights, which will be rigorously derived in the remainder of the article, are summarized as follows:
\begin{itemize}
\item The convergence rate, characterized by $\rho_\text{c}$ in \eqref{eq:conv1}, is uniquely determined by the \emph{local} shape of the objective function about an isolated minimum. The global shape of the objective function determines the constant $C_\text{c}$ and the set of initial conditions $x(0)$ for which \eqref{eq:conv1} holds. In other words, the global shape (for example described by convexity) determines the stability and region of attraction of a local minimum, whereas the local shape, i.e., the curvature at an isolated minimum, determines the convergence rate.

\item An algorithm is called accelerated if the convergence rate scales favorably for ill-conditioned optimization problems, meaning that the convergence rate scales favorably for large $\kappa$. Acceleration is therefore a statement about the robustness of the convergence rate with respect to changes in the curvature, which is well defined for continuous-time as well as for discrete-time formulations.

\item Accelerated convergence is generic to momentum-based optimization algorithms, provided that the damping scales with $1/\sqrt{\kappa}$ for large $\kappa$. Neither the evaluation of the gradient at a shifted position, nor a specifically engineered damping parameter, as for example proposed in \citet[Sec.~2.2]{NesterovIntro}, are necessary. 

\item From a physics perspective, a momentum-based optimization algorithm can be thought of as a mass-spring-damper system, where the spring potential is given by the objective function. The algorithm design specifies the damping. The fact that the system has inertia (due to the second-order dynamics) implies that the convergence rate is robust to small changes in the spring, i.e., robust to small changes in the curvature of the objective function. The inertia gives the system the tendency to keep its velocity, which, provided that the damping is chosen appropriately, implies that small changes in the spring will not slow down convergence. This captures the mechanism behind accelerated convergence.

\item The convergence rate typically becomes arbitrarily small for large $\kappa$ (if the dynamics are time-varying this statement applies for $t\rightarrow \infty$). The underlying dynamics are therefore close to conservative, which means that great care is required for the discretization. A discretization that introduces an artificial energy drift, such as the explicit forward Euler discretization, for example, might even lead to instability, which clearly makes a favorable scaling of the convergence rate with $\kappa$ impossible. This motivates the use of symplectic discretization schemes, which preserve the Hamiltonian structure of the underlying conservative part of the dynamics.

\item By introducing time-varying dynamics, which are obtained by adjusting the damping parameters with the number of iterations, the linear convergence rate is improved with an additional sublinearly converging term. If the local minimum is close to degenerate, the convergence is dominated by this term, as the rate of the linear convergence becomes arbitrarily small. In that way, the analysis relates the non-degenerate and degenerate cases.
\end{itemize}

\noindent\emph{Related work:} The phenomenon of acceleration, which fundamentally motivates the use of momentum, has puzzled many researchers for almost 40 years. We do not attempt to give a full overview of the literature, but highlight some of the most recent work. Important contributions were made by \citet{Bubeck}, who proposed an accelerated algorithm that has a geometric interpretation, by \citet{Allen-Zhu}, who show that coupling of gradient and mirror descent can lead to acceleration, and \citet{LessardRecht}, who propose a general control-theoretic analysis framework. The framework has subsequently been extended and refined, for example by \citet{Scherer}, who analyzed and quantified robustness and convergence trade-offs. Other work includes \citet{Diakonikolas}, who unify the analysis of first-order methods by imposing certain decay conditions, and \cite{Scieur}, who interpret the accelerated gradient method as a multi-step discretization of gradient flow. 

The work of \citet{SuAcc} and \citet{KricheneAcc} showed that the trajectories of the accelerated gradient method approach the solutions of a certain second-order ordinary differential equation. The resulting differential equation was analyzed in further detail in \citet{Attouch} and placed within a variational framework by \citet{WibisonoVariational}. This motivated further research on structure-preserving integration schemes for discretizing continuous-time optimization algorithms \citep{Betancourt}. While the continuous-time formulation of \citet{SuAcc} is based on an accelerated optimization algorithm for smooth and convex objective functions (i.e., the case where the minimum could be degenerate), an alternative for strongly convex objective functions (non-degenerate case) was proposed in \cite{Ebenbauer}. In \citet{Gui} and \citet{ourWork}, important geometric properties of the underlying dynamics are highlighted and corresponding structure-preserving discretization schemes are analyzed. In addition, \citet{Gui} propose relativistic dynamics, as these naturally bound the rate of change of the position by the speed of light, which is supposed to prevent overshoot. The work of \citet{HDM} suggests that convergence can be improved by a suitable choice of the kinetic energy (for example choosing the kinetic energy to be the convex conjugate of the objective function).

While this line of work suggests strong parallels between dynamical systems and optimization algorithms, a main part of the research focuses on engineering Lyapunov functions or estimate sequences that arise separately from the dynamical system. These typically arise from educated guesses, or from the solutions of semidefinite programs, and enable the precise characterization of the convergence rate for a given algorithm. The aim of our work is different: We exploit fundamental topological properties of dynamical systems in order to provide an explicit characterization of the convergence rate up to constants, which, in most cases, can be evaluated in closed form. Hence, our analysis does not aim at precisely bounding the number of iterations needed to reach a local minimum with a certain tolerance, but provides a qualitative and quantitative characterization of the convergence rate up to constants. The analysis does not rely on convexity and applies both to continuous-time and discrete-time formulations. This provides a tool to understand how the convergence rate depends on various algorithm parameters and the shape of the objective function, and it characterizes the mechanism leading to acceleration. Our characterization of the convergence rate is both necessary and sufficient.

In the context of nonconvex optimization, the aim is typically to find a local minimum of a twice continuously differentiable function that satisfies certain non-degeneracy conditions. It has been shown that gradient descent converges to a local minimum from almost every initial condition, which is due to the fact that local maxima and saddle points are unstable equilibria \citep{LeeGD}. The same reasoning applies to gradient descent with momentum. However, even though gradient descent and gradient descent with momentum are guaranteed (in an almost everywhere sense) to ultimately reach a local minimum, it may take an arbitrarily long time to escape saddle points and local maxima. A major concern arises due to the fact that an objective function with $n$ decision variables might have $n-1$ different isolated saddle points, which have to be traversed before reaching a local minimum. This implies that even for generic initialization strategies, the worst-case convergence rate depends on $n$ \citep{SimonDu}. This article is concerned with characterizing the convergence of momentum-based algorithms up to a constant factor. We will not analyze how this factor scales with $n$; the example from \cite{SimonDu} suggests that the factor scales exponentially in $n$. However, the results from \citet{Ge} indicate that adding small random perturbations to gradient descent can improve the convergence rate and reduce the dimension-dependence of the convergence rate on $n$ to polylog factors. A similar result applies likewise to gradient descent with momentum \citep{ChiAcc}. A recent account of the state-of-the art is given in \citet{Chi}, for example.

The poor scaling in the dimension $n$ in the nonconvex case can be avoided by explicitly leveraging curvature information, as for example proposed in \citet{NesterovCubic} and \citet{Curtis}. However, the computational cost per iteration of these methods is generally higher and increases with $n$. For simplicity, this article focuses on first-order algorithms with momentum, even though many ideas generalize in straightforward ways.

\medskip
\noindent\emph{Outline:} 
This article views optimization algorithms as dynamical systems. Two cases are distinguished: Section~\ref{Sec:TI} assumes that the algorithm parameters are fixed, leading to time-invariant dynamics. The more general case, where the parameters are allowed to vary with the number of iterations, is discussed in Section~\ref{Sec:TV}.

Section~\ref{Sec:TI} starts by introducing the notation and defining the scope of the analysis. A momentum-based optimization algorithm is understood as a second-order dynamical system,\footnote{In this context, the term ``second order'' has nothing to do with accessing higher derivatives of the objective function.} where the local minima of the objective function correspond to asymptotically stable equilibria. Section~\ref{Sec:Ex1} and Section~\ref{Sec:Ex2} introduce two prototypical examples of momentum-based optimization algorithms. These are generalized versions of Nesterov's acclerated gradient scheme \citep[Ch.~2.2]{NesterovIntro}, and also include Heavy-Ball methods \citep{PolyakHeavyBall}. \hchange{By focusing on smooth dynamical systems based on smooth objective functions, our analysis excludes, for example, the treatment of constraints via (non-differentiable) indicator functions, as well as optimization algorithms with restart schemes.} The subsequent result derived in Section~\ref{Sec:CharTI} highlights our assertion that fundamental topological properties can be exploited for characterizing the convergence rate of momentum-based optimization algorithms. The results are applied to the analysis of the two prototypical examples in Section~\ref{Sec:OptiEx1} and Section~\ref{Sec:OptiEx2}, leading to a broad characterization of the phenomenon of acceleration. In addition, Section~\ref{Sec:OptiEx2} rigorously motivates the use of symplectic discretization schemes in the context of optimization: The symplectic discretization enables the computation of a modified energy function that can be used for stability analysis.

The structure of Section~\ref{Sec:TV} is analogous to Section~\ref{Sec:TI}. The general result for characterizing the convergence rate in the time-varying case is presented in Section~\ref{Sec:CharRateTV} and illustrated with two subsequent examples in Section~\ref{Sec:OptiEx1TV} and Section~\ref{Sec:OptiEx2TV}. The results highlight the fact that time-varying damping parameters can speed up the convergence rate by an additional sublinearly converging factor. For ill-conditioned problems, this additional gain in the convergence rate becomes significant. The analysis connects the non-degenerate case with the degenerate case and motivates the update rules for the parameters of well-known accelerated gradient schemes, such as \citet[p.~90, Constant Step Scheme II]{NesterovIntro}.

The article concludes with a summary and final remarks in Section~\ref{Sec:Conclusion}.



\section{The Time-Invariant Case}\label{Sec:TI}
\subsection{Introduction}
Throughout the article we consider the problem of minimizing the function $f: \mathbb{R}^n\rightarrow \mathbb{R}$, which satisfies the following assumption:
\begin{assumption}\label{Ass:funfund}
The function $f$ has a Lipschitz-continuous gradient. The critical points are non-degenerate and isolated.\footnote{A critical point is non-degenerate if the Hessian $\diff^2 f/\diff x^2$ is nonsingular for all $x\in\mathbb{R}^n$ (almost everywhere) in a neighborhood of the critical point.}
\end{assumption}
The Lipschitz continuity of the gradient is important for ensuring that the resulting continuous-time trajectories exist and are unique. The fact that $f$ has isolated non-degenerate critical points excludes pathological cases where the function has multiple connected local minima. The less pathological case, where the local minima are isolated but have vanishing curvature in certain directions, can be obtained via a limit argument, as will be discussed in Section~\ref{Sec:TV}. 

Due to the Lipschitz continuity of the gradient, the Hessian exists almost everywhere and is essentially bounded. We will summarize the (essential) upper and lower bounds on the Hessian with the two constants $C_\text{f}\geq 0$ and $\bar{C}_\text{f}\geq 0$:
\begin{equation}
-C_\text{f} |u|^2 \leq u\T \frac{\partial^2 f}{\partial x^2} u \leq \bar{C}_\text{f} |u|^2, \quad\forall u\in \mathbb{R}^n, \forall x\in \mathbb{R}^n ~\text{(a.e.)}, \label{eq:defCf}
\end{equation}
where $|\cdot|$ denotes the Euclidean norm. Thus for $C_\text{f}=0$ the function is convex, for $C_\text{f}>0$ it is nonconvex.

In order to simplify our exposition, we will consider functions that satisfy:
\begin{assumption}\label{Ass:fun}
In addition to Assumption~\ref{Ass:funfund}, $f$ has a second derivative that is Lipschitz continuous.
\end{assumption}
None of the subsequent results will explicitly depend on a Lipschitz constant related to the second derivative. Hence, under very mild conditions, all our results characterizing the convergence rate $\rho_\text{c}(t)$ in \eqref{eq:conv1} apply when $f$ satisfies Assumption~\ref{Ass:funfund} instead of Assumption~\ref{Ass:fun}.\footnote{However, the smoothness of $f$ might affect the constant $C_\text{c}$ in \eqref{eq:conv1}. The important difference to the higher-order smoothness assumptions made in \citet{NesterovCubic} or \citet{BubeckHigherOrder}, for example, is that Assumption~\ref{Ass:fun} does not impose \emph{predefined} bounds on the higher derivatives of $f$, which would restrict the shape of $f$.} This is further discussed in Appendix~\ref{App:Smooth}.

Without loss of generality, we further assume that $f$ has a local minimum at $x^*=0$ and that $f(x^*)=f(0)=0$. The convergence rate of a momentum-based optimization algorithm will depend on the local shape of the local minimum in question, which is determined by the constants
\begin{equation}
\mu:=\min_{u\in \mathbb{R}^n, |u|^2=1} u\T \left.\frac{\partial^2 f}{\partial x^2}\right|_{x=x^*} u, \quad L:=\max_{u\in \mathbb{R}^n, |u|^2=1} u\T \left.\frac{\partial^2 f}{\partial x^2}\right|_{x=x^*} u, \quad \kappa:=\frac{L}{\mu}. \label{eq:defmul}
\end{equation}
In case Assumption~\ref{Ass:fun} is replaced with Assumption~\ref{Ass:funfund}, the above constants are defined via the essential supremum and essential infimum of the Hessian in a neighborhood of $x^*$.

We model a momentum-based optimization algorithm either as a continuous-time or discrete-time dynamical system of the form
\begin{gather}
q^+(t)=g_{\text{q}}(q(t),p(t)), \quad
p^+(t)=g_{\text{p}}(q(t),p(t)), \quad \forall t\in I, \label{eq:ODE11}\\
q(0)=q_0, \quad p(0)=p_0, \label{eq:ODE12}
\end{gather}
where the superscript $+$ denotes either differentiation with respect to $t$ (continuous-time setting), in which case $I=\mathbb{R}_{\geq 0}$, or a unit time-shift (i.e., $q^+(t)=q(t+1)$, discrete-time setting), in which case $I=\{0,1,2,\dots\}$. The nonnegative real numbers are denoted by $\mathbb{R}_{\geq 0}$, whereas the positive real numbers are denoted by $\mathbb{R}_{>0}$. The dynamics
\begin{equation}
g_\text{q}: \mathbb{R}^n \times \mathbb{R}^n \rightarrow \mathbb{R}^n, \quad g_\text{p}: \mathbb{R}^n \times \mathbb{R}^n \rightarrow \mathbb{R}^n,
\end{equation}
are implicitly dependent on $\nabla f$ and are assumed to satisfy the following assumption.\footnote{In the same way that the additional assumptions on the Hessian of $f$ can be relaxed to mere Lipschitz continuity of $\nabla f$, Assumption~\ref{Ass:Cont} can be relaxed to Lipschitz continuity of $g_\text{q}$ and $g_\text{p}$; cf.\ Appendix~\ref{App:Smooth}.}
\begin{assumption}\label{Ass:Cont}
The dynamics $g_\text{q}$ and $g_\text{p}$ are continuously differentiable in both arguments and the derivatives are Lipschitz continuous.
\end{assumption}

In the continuous-time case, Assumption~\ref{Ass:Cont} implies that the resulting trajectories exist and are unique for all times $t\in I$ \citep[p.~93, Corollary~3]{DifferentialEquationsArnold}. Differentiability also implies that the dynamics can be linearized about an equilibrium, which typically provides a means to study the local behavior of the resulting trajectories. In addition, Assumption~\ref{Ass:Cont} implies that, over a finite time interval, trajectories are continuously dependent on their initial conditions \citep[p.~93, Corollary~4]{DifferentialEquationsArnold}. These topological properties will be exploited in the following.



In order to simplify notation, we define 
\begin{equation}
g: \mathbb{R}^{2n} \rightarrow \mathbb{R}^{2n}, \quad g:=(g_\text{q},g_\text{p}),
\end{equation} 
and introduce $z(t):=(q(t),p(t))$ for all $t\in I$. Moreover, the map  $(q_0,p_0) \rightarrow (q(t),p(t))$ is denoted by $\varphi_t: \mathbb{R}^{2n} \rightarrow \mathbb{R}^{2n}$, $t\in I$, and is referred to as the flow of the dynamical system \eqref{eq:ODE11} - \eqref{eq:ODE12}.

Next, we provide a formal definition of a momentum-based optimization algorithm.
\begin{definition}\label{def:MomBasedOA}
We call the dynamical system \eqref{eq:ODE11} - \eqref{eq:ODE12} satisfying Assumption~\ref{Ass:Cont} a \emph{momentum-based optimization algorithm} for the function $f$ if $x^*=0$ is an asymptotically stable equilibrium in the sense of Lyapunov. 

This means that $\varphi_t(0)=0$, for all $t\in I$, $\varphi_t$ is continuous at $0$, uniformly in $t$, and $\lim_{t\rightarrow \infty} \varphi_t(z_0)=0$ for any $z_0$ in a neighborhood of the origin.
\end{definition}
\hchangeII{Even though the dynamics $g_\text{q}$ and $g_\text{p}$ can capture gradient flow, as a special case (e.g., $g_\text{p}=0$, $g_\text{q}=-\nabla f(q)$ in continuous time), we are interested in analyzing momentum methods, which arise from a nontrivial choice of $g_\text{q}$ and $g_\text{p}$.} 

Throughout the article we will illustrate our ideas with two examples, which are prototypical versions of momentum-based optimization algorithms. 

\subsection{Example 1}\label{Sec:Ex1}
The first example is based on the following continuous-time dynamics
\begin{equation}
\dot{q}(t)=p(t), \qquad \dot{p}(t)=-\nabla f(q(t)) + f_\text{d}(q(t),p(t)), \qquad \forall t\in \mathbb{R}_{\geq 0}, \label{eq:exCT}
\end{equation}
where $\nabla f$ denotes the gradient of $f$ and $f_\text{d}$ the \hchange{dissipative forces}. These dynamics can be viewed as a mass-spring-damper system, where $f$ represents the spring potential and $f_\text{d}$ the damping. The \hchange{dissipative forces} are assumed to take the form
\begin{equation}
f_\text{d}(q,p):=-2d p-(\nabla f(q+\beta p)-\nabla f(q)), \label{eq:fnp}
\end{equation}
where the parameters $d>0, \beta \geq 0$ are constant. Ideally, the parameters $d$ and $\beta$ are designed to take into account additional information about the the function $f$; for example, upper and lower bounds on the curvature $\diff^2 f / \diff x^2$. If $\beta$ is chosen to be zero, the \hchange{dissipative forces} reduce to $-2d p$. In that case, the dynamics \eqref{eq:exCT} describe a continuous-time heavy ball method \citep{PolyakHeavyBall}. In case $\beta >0$, the dynamics are related to Nesterov's accelerated gradient method \citep{ourWork}.

An intuitive interpretation of the \hchange{dissipative forces} \eqref{eq:fnp} can be given in the following way: For $\beta=0$, \eqref{eq:fnp} describes linear isotropic damping. For $\beta >0$, \eqref{eq:fnp} can be rewritten as
\begin{equation}
-2d p -\int_{0}^{\beta} \left.\frac{\diff^2 f}{\diff x^2}\right|_{q + \tau p} \diff \tau~ p,
\end{equation}
which implies that \eqref{eq:fnp} includes an additional damping term that averages the local curvature in the interval between $q$ and $q + \beta p$. As a result, the damping increases if the local curvature is large, and reduces if the local curvature is small. As the velocity $p$ is larger, the interval over which the average is taken is increased. The two forms of damping, linear isotropic and curvature dependent, are balanced by the coefficients $d$ and $\beta$.

The equilibria of \eqref{eq:exCT} are given by the critical points of $f$. Moreover, if a given critical point is a non-degenerate local minimum, the corresponding equilibrium is asymptotically stable. This follows by evaluating the total energy, 
\begin{equation}\label{eq:energy}
H(q,p)=\frac{1}{2} |p|^2 + f(q),
\end{equation}
along the trajectories of \eqref{eq:exCT},
\begin{equation}
\frac{\diff}{\dt} H(q(t),p(t))=f_\text{d}(q(t),p(t))\T p(t) = -2d |p|^2 - p\T \int_{0}^{\beta} \left.\frac{\diff^2 f}{\diff x^2}\right|_{q + \tau p} \diff \tau~ p, \label{eq:energyDec}
\end{equation}
which shows that the energy necessarily decreases in a neighborhood of the equilibrium. Combined with the fact that $H$ is locally positive definite about a non-degenerate local minimum, this implies stability in the sense of Lyapunov. Asymptotic stability of the non-degenerate local minimum can then be concluded from La Salle's theorem \citep[see, for example,][Ch.~5.4]{Sastry}, which is based on examining the invariant sets satisfying $\diff H/\diff t=0$. Thus, the dynamical system \eqref{eq:exCT} is a momentum-based optimization algorithm for $f$ according to Definition~\ref{def:MomBasedOA}.

However, analyzing how the energy evolves along the trajectories of \eqref{eq:exCT} also reveals global properties of the dynamics. For $\beta=0,d>0$ it follows that the set of initial conditions that do not converge to a local minimum is a set of measure zero, given by the critical points of $f$ that are not local minima. The analysis holds without assuming convexity of $f$. However, if $f$ happens to be convex and has a unique global minimum, then the corresponding equilibrium is globally asymptotically stable for any choice of parameters $\beta\geq 0, d>0$. 

The same reasoning applies in case $f$ satisfies Assumption~\ref{Ass:funfund} instead of \ref{Ass:fun}, or when $f$ has degenerate local minima.

\subsection{Example 2}\label{Sec:Ex2}
The second example is obtained by discretizing \eqref{eq:exCT} in the following way:
\begin{equation}
q_{k+1}=q_k+T p_{k+1}, \qquad p_{k+1}=p_k - T \nabla f(q_k) + T f_\text{d}(q_k,p_k), \qquad \forall k\in \{0,1,\dots\},\label{eq:exDT}
\end{equation}
where $T>0$ is the step size. \hchange{The discretization consists of a forward Euler update of the momentum coordinates, and uses the newly computed momentum for the position update. As will be further discussed in the remainder of the article, the fact that the newly computed momentum coordinate is used for the position update makes the scheme symplectic for $f_\text{d}=0$. This means that the transformation \eqref{eq:exDT} from $(q_k,p_k) \rightarrow (q_{k+1},p_{k+1})$ preserves the symplectic form (for $f_\text{d}=0$), which has important consequences. One of these consequences concerns the spectrum of the linearization of \eqref{eq:exDT}. In case $f_\text{d}=0$, the corresponding eigenvalues are guaranteed to lie on the unit circle (for $T\leq 2/\sqrt{L}$). This is in sharp contrast to the standard explicit Euler discretization, which is not symplectic, and where the eigenvalues lie outside the unit circle even for arbitrarily small $T>0$. Indeed, we will exploit the fact that the map \eqref{eq:exDT} is symplectic (for $f_\text{d}=0$) to construct a modified energy function, which will be used for a stability analysis that extends beyond the linearization. 
Additional background information on symplectic integration can be found in \cite{Serna} or \cite{Hairer}, for example.} 

For $\beta=0$, the resulting algorithm is referred to as gradient descent with momentum \citep{PolyakHeavyBall}. For $\beta>0$, Nesterov's accelerated gradient scheme \citep[Constant step scheme III, p.~81]{NesterovIntro} is obtained by choosing the parameters as follows:
\begin{equation}\label{eq:NParams}
T=\frac{1}{\sqrt{L}}, \quad d=\frac{\sqrt{L}}{\sqrt{\kappa}+1}, \quad \beta=\frac{\sqrt{\kappa}-1}{\sqrt{\kappa}+1}\frac{1}{\sqrt{L}},
\end{equation}
where $\kappa$ and $L$ characterize the local shape of a local minimum, \citep{ourWork}. Moreover, as pointed out in \cite{ourWork}, the ``Constant step scheme II" algorithm of \citet{NesterovIntro} is obtained by a particular choice of time-varying coefficients $\beta$ and $d$. The generalization to time-varying coefficients will be discussed in Section~\ref{Sec:CharRateTV}.

The equilibria of \eqref{eq:exDT} are again the stationary points of $f$. In order to determine the stability of an isolated local minimum, we linearize the dynamics,
\begin{equation}
\delta q_{k+1}=\delta q_k + T \delta p_{k+1}, \qquad \delta p_{k+1}=\delta p_k - 2dT \delta p_k - \beta T H_\text{e} \delta p_k - T H_\text{e} \delta q_k,
\end{equation}
where $H_\text{e}:=\diff f/\diff x^2|_{x=0}$ (without loss of generality we consider $x^*=0$). An eigenvalue analysis then reveals that the corresponding non-degenerate local minimum is asymptotically stable if
\begin{equation}\label{eq:DTlinCond}
0< T(2d+\beta h) < 2- h T^2/2
\end{equation}
holds for all eigenvalues $h$ of $H_\text{e}$, i.e. $\mu \leq h \leq L$.\footnote{The condition \eqref{eq:DTlinCond} is necessary and sufficient for the the linearized dynamics to be asymptotically stable for all $h$, with $\mu \leq h \leq L$.} Asymptotic stability of the linearized dynamics implies that the same equilibrium is asymptotically stable under the nonlinear dynamics \eqref{eq:exDT} \citep[p.~215]{Sastry}.

For example, provided that the parameters $d$ and $\beta$ are chosen according to \eqref{eq:NParams}, we obtain that the given equilibrium is asymptotically stable if 
\begin{equation}
0< T < (\sqrt{5}-1)/\sqrt{L}.
\end{equation}
We thus conclude that \eqref{eq:exDT} is an optimization algorithm for $f$ according to Definition~\ref{def:MomBasedOA} provided that the constants $d$, $\beta$, and $T$ are chosen such that \eqref{eq:DTlinCond} is satisfied.

Unlike Example 1, our analysis for Example 2 is valid only in a neighborhood about the equilibrium, which corresponds to a local minimum. However, the specific structure of the discretization and the ideas from Example 1 can be exploited for obtaining a nonlinear analysis that is valid beyond a neighborhood of the equilibrium. This will be illustrated in Section~\ref{Sec:OptiEx2}.

\subsection{Characterizing the convergence rate}\label{Sec:CharTI}
Guaranteeing mere convergence is often not enough, as we are primarily interested in how quickly the trajectories of the nonlinear dynamics \eqref{eq:ODE11} and \eqref{eq:ODE12} converge to a local minimum. In the following, we argue that in most cases, a linear analysis characterizes the convergence rate up to constants. The linear analysis typically reduces to the computation of eigenvalues. 

Our main proposition is based on the following assumption and definition.
\begin{assumption}
Let the linearized dynamics
\begin{align*}
\delta z^+(t)=\left.\frac{\partial g}{\partial z}\right|_{z=0} \delta z(t), \quad \forall t\in I, \quad \delta z(t_0)=z_0,
\end{align*}
be such that there exists an estimate
\begin{equation*}
|\delta z(t)|\leq C_\text{l} |\delta z(t_0)| \exp(-\alpha (t-t_0)), \quad \forall z_0\in \mathbb{R}^{2n}, \quad \forall t_0, t \in I, t\geq t_0,
\end{equation*}
where $C_\text{l}\geq 1$ and $\alpha>0$ are constant.
\label{Ass:lin}
\end{assumption}

\begin{definition}
The region of attraction of the equilibrium at the origin (of the nonlinear dynamics \eqref{eq:ODE11}) is defined as the set 
\begin{equation*}
\mathcal{R}:=\{ z_0\in \mathbb{R}^{2n}~|~\lim_{t\rightarrow \infty} \varphi_t(z_0)=0\}.
\end{equation*}
\end{definition}


\begin{proposition}
\label{Prop:main}
Let Assumption~\ref{Ass:lin} be satisfied. Then, for any compact set $A\subset \mathcal{R}$ there exists a finite constant $\hat{C}\geq 1$ such that for all $z_0 \in A$
\begin{equation}\label{eq:boundtmp}
|\varphi_t(z_0)| \leq \hat{C} |z_0| \exp(-\alpha t), \quad \forall t\in I.
\end{equation}
\end{proposition}

\begin{proof}
If $A$ is empty the statement is trivial. We therefore assume that $A$ is non-empty. Lemma~\ref{Lem:AsymptConv} (see Appendix~\ref{App:PropMain}) implies that there exists an open ball $B_\delta$ of radius $\delta > 0$, centered at the origin, such that any trajectory starting in  $B_\delta$ converges with rate $\alpha$. We make the following claim.

\noindent\emph{Claim:} There exists a finite time $T_\text{m}>0$, $T_\text{m}\in I$ such that for all $z_0\in A$, $\varphi_{t}(z_0) \in B_\delta$, for all $t\in I$, $t\geq T_\text{m}$.

\noindent\emph{Proof of the claim:} The origin is a stable equilibrium. Hence there exists a constant $\epsilon>0$ such that all trajectories starting in $B_\epsilon$, the open ball of radius $\epsilon$ centered at the origin, remain in $B_\delta$ for all times. In addition, each $z_0\in A$ satisfies $\lim_{t\rightarrow\infty} \varphi_t(z_0) =0$, since $A\subset \mathcal{R}$. Thus, for each $z_0 \in A$ there exists a time $T(z_0)$ such that $\varphi_{T(z_0)}(z_0)<\epsilon/2$. The continuity assumptions on $g$ imply that $\varphi_{T(z_0)}$ is continuous, which can be verified by the Gr\"{o}nwall inequality. Therefore, for each $z_0\in A$ there exists an open ball $B(z_0)$ centered about $z_0$ such that for all $\bar{z}_0 \in B(z_0)$, $\varphi_{T(z_0)}(\bar{z}_0)<\epsilon$, which implies $\varphi_{t}(\bar{z}_0) \in B_\delta$ for all $t\geq T(z_0)$. The collection of all $B(z_0), z_0\in A$ is an open cover for $A$. Due to the fact that $A$ is compact, there exists a finite sub-cover (according to the Heine-Borel theorem), which we denote $B(\hat{z}_{i})$, $i=1,2,\dots, N$, where $N$ is finite. As a result, choosing $T_\text{m}:=\max_{i\in \{1,2,\dots, N\}} T(\hat{z}_{i})$ implies $\varphi_{t}(z_0) \in B_\delta$ for all $z_0\in A$, $t\geq T_\text{m}$, which proves the claim.

Moreover, the continuity of $\varphi_t$ for all $t\in I$, $0\leq t \leq T_\text{m}$, implies further that $\varphi_t(A)$ is bounded for any $t\in I$, $0\leq t \leq T_\text{m}$. Combined with Lemma~\ref{Lem:AsymptConv} (see Appendix~\ref{App:PropMain}), this yields the following bound
\begin{equation}
|\varphi_t(z_0)|\leq 
\begin{cases} C_\text{A} &0\leq t \leq T_\text{m}, t\in I,\\
\delta \tilde{C} \exp(-\alpha (t-T)) &t>T_\text{m}, t\in I,
\end{cases}\label{eq:prooftmp11111}
\end{equation}
for all $z_0\in A$, where $C_\text{A}\geq\delta$ and $\tilde{C}\geq 1$ are positive constants. We fix $z_0\in A$, consider the trajectory $z(t):=\varphi_t(z_0)$, $t\in I$, and apply the mean value theorem,
\begin{equation}
z^+(t)=\left(\left.\frac{\partial g}{\partial z}\right|_{z=0} + \left.\frac{\partial g}{\partial z}\right|_{z=\xi(t)} - \left. \frac{\partial g}{\partial z}\right|_{z=0} \right) z(t),
\end{equation}
where $\xi(t)$ lies between $z(t)$ and the origin. Due to the fact that the dynamics are assumed to have Lipschitz-continuous derivatives, we obtain the following bound:
\begin{equation}\label{eq:prooftmp1}
\left\vert \left.\frac{\partial g}{\partial z} \right|_{z=\xi(t)} - \left.\frac{\partial g}{\partial z}\right|_{z=0} \right\vert \leq \bar{C}_A |\xi(t)| \leq \bar{C}_A |z(t)|,
\end{equation}
where $\bar{C}_A$ denotes a Lipschitz constant of $\partial g/\partial z$ on $A$. According to \eqref{eq:prooftmp11111}, the trajectory $|z(t)|$ is integrable (in continuous time) and absolutely summable (in discrete time). We obtain, by virtue of Lemma~\ref{Lem:Gronwall} (see Appendix~\ref{App:PropMain}),
\begin{equation}
|z(t)|\leq C_\text{l} \exp(C_\text{l} \bar{C}_A C_\text{z}) |z_0| \exp(-\alpha t),  \quad \forall z_0 \in A, \quad \forall t\in I,
\end{equation}
where $C_\text{z}$ is constant. The constant $C_\text{z}$ is related to an upper bound on the integral (in continuous time) or the sum (in discrete time) of $|z(t)|$ over $t\in I$, which according to \eqref{eq:prooftmp11111}, is guaranteed to be finite.
\end{proof}

\noindent\emph{Remarks:}
\begin{itemize}
\item Proposition~\ref{Prop:main} characterizes the convergence rate and states that the number of iterations required to obtain an $\epsilon$-accuracy approaches $\log(1/\epsilon)/\alpha$ for small $\epsilon$. This does not provide a tight bound on the number of iterations required to achieve a certain accuracy, as the constant $\hat{C}$ might depend on $\alpha$ or grow rapidly with the size of $A$. Nevertheless, it enables a qualitative and quantitative discussion of the convergence rate $\alpha$. In particular Proposition~\ref{Prop:main} highlights the fact that the convergence rate $\alpha$ is determined by the local properties of the dynamics, which depend on the local shape of the objective function $f$.
\item The assumptions required for invoking Proposition~\ref{Prop:main} are often straightforward to verify, as Assumption~\ref{Ass:lin} hinges on an eigenvalue analysis of $\partial g/\partial z$ at $z=0$. Proposition~\ref{Prop:main} can also be generalized to the case where $f$ satisfies Assumption~\ref{Ass:funfund} instead of Assumption~\ref{Ass:fun}, as shown in Appendix~\ref{App:Smooth}.
\item \hchangeII{The proof highlights the following alternative statement of Proposition~\ref{Prop:main}: There exists a finite time $T_\text{m}>0$ after which all trajectories starting in $A$ are guaranteed to be within a small neighborhood of the origin. Within this neighborhood the convergence is exponential with rate $\alpha$.}
\item Given the smoothness properties of $g$, the assumption that the linearized dynamics converge with rate $\alpha >0$ (Assumption~\ref{Ass:lin}) is necessary and sufficient for the convergence of the nonlinear dynamics with rate $\alpha$, as can be shown with the arguments of Lemma~\ref{Lem:AsymptConv} in Appendix~\ref{App:PropMain}.
\end{itemize}

\subsection{Implications for optimization algorithms}
Proposition~\ref{Prop:main} enables a characterization of the convergence rate based on a linear analysis of the dynamics about an equilibrium. In the following, we will use Proposition~\ref{Prop:main} to discuss the convergence rate of the optimization algorithms given in Example 1 and Example 2. In particular, this provides conditions guaranteeing accelerated convergence.

\subsection{Example 1}\label{Sec:OptiEx1}
The total energy of the dynamical system is given by \eqref{eq:energy}, and according to \eqref{eq:energyDec}, energy is dissipated along trajectories. As we will show in the following, the energy function can therefore be used to characterize the region of attraction of the equilibrium $z^*=0$, whereby the topology of the level sets of $H$ and $f$ will play an important role. This will be discussed next.

By assumption, $f$ has a local non-degenerate minimum at $x^*=0$, which implies that $f^{-1}([0,c])$ contains a compact, connected component containing the origin for sufficiently small $c>0$. The set $f^{-1}([0,c])$ describes all values $x\in \mathbb{R}^n$ such that $0\leq f(x) \leq c$. Morse theory~\citep{MorseTheory} concludes that the topology of the set $f^{-1}([0,c])$ is determined by the critical points of $f$. Thus, the set $f^{-1}([0,c])$ includes a compact, connected component that contains the origin (and no other critical point), as long as $c<\hat{f}=f(\hat{x})$, where $\hat{x}$ is any other critical point. The situation is illustrated with an example in Figure~\ref{Fig:energy}.

\begin{figure}
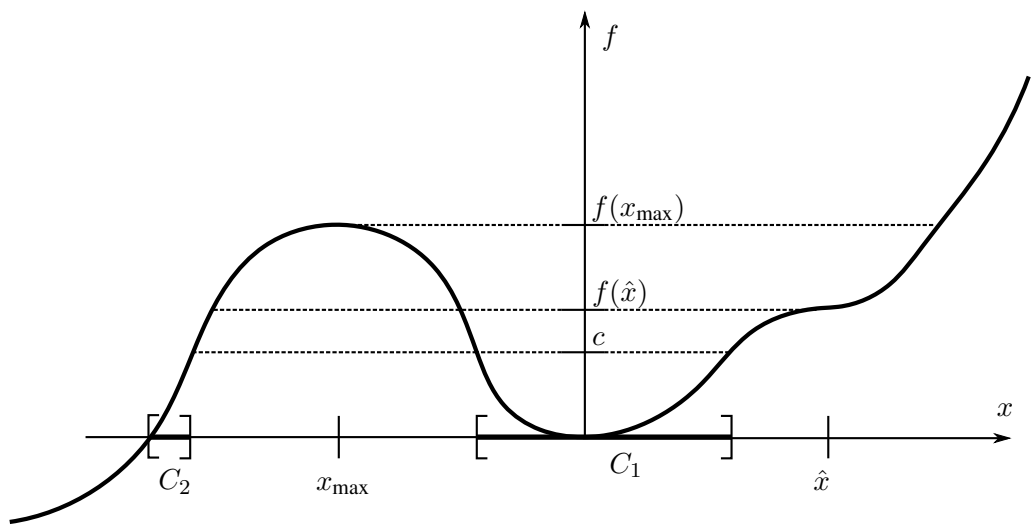
\caption{The figure illustrates the set $f^{-1}([0,c])$, with $c>0$, when the function $f$ is scalar. In this example, $f^{-1}([0,c])$ contains two connected components $C_1$ and $C_2$. The function $f$ has a local minimum at the origin, a saddle point at $\hat{x}$, and a local maximum at $x_\text{max}$. The origin is contained in the set $C_1$. Morse theory states that $C_1$ is guaranteed to be compact provided that $0<c< f(\hat{x})$. Due to the fact that $f$ is one-dimensional $C_1$ remains compact for $0<c<f(x_\text{max})$; this is, however, no longer true in higher dimensions.}
\label{Fig:energy}
\end{figure}

\hchange{From the definition of $H$ we infer that the critical points of $H$ are all of the form $(q^*,0)$, where $q^*$ corresponds to a critical point of $f$. The above reasoning therefore also applies to the total energy $H$, and concludes that the set $H^{-1}([0,c])$ includes a compact, connected component that contains the origin (and no other critical point), as long as $c<\hat{f}$.} This motivates the following definition, which will be used throughout the remainder of the article.
\begin{definition}\label{def:SetA}
The set $\mathcal{A}_f$ is defined as the connected component of $H^{-1}([0,\hat{f}))$ that contains the origin.
\end{definition}

Analyzing the rate of change of the energy along the trajectories of \eqref{eq:exCT} (cf.\ \eqref{eq:energyDec}),
\begin{equation}
\frac{\diff}{\diff t} H(q(t),p(t)) = - 2d |p(t)|^2 - p(t)\T \int_{0}^{\beta} \left. \frac{\diff^2 f}{\diff x^2}\right|_{q(t) + \tau p(t)} \diff \tau ~p(t),\label{eq:energyDec2}
\end{equation}
reveals that by a suitable choice of the parameters $d$ and $\beta$, the energy necessarily decays. 
In particular, this is the case for $d>0$, $0\leq \beta <2d/C_\text{f}$, where $C_\text{f}\geq 0$ denotes a lower bound on the Hessian of $f$, as defined in \eqref{eq:defCf}. We therefore conclude as follows.
\begin{proposition}
Provided that $d>0$ and $0\leq \beta \leq 2d/C_\text{f}$, the origin is an asymptotically stable equilibrium in the sense of Lyapunov. Its region of attraction contains the set $\mathcal{A}_f$.
\end{proposition}
\begin{proof}
We consider any initial condition $z(0)\in \mathcal{A}_f$ and define $H_0:=H(z(0))$. For $d>0$ and $0\leq \beta \leq 2d/C_\text{f}$, the energy necessarily decays along the trajectories of \eqref{eq:exCT}, c.f. \eqref{eq:energyDec2}. The trajectory $z(t)=(q(t),p(t))$ is therefore confined to the connected component of $H^{-1}([0,H_0])$ that contains the origin, which, according to the above discussion, is necessarily compact. This implies that the origin is stable. In addition, the energy strictly decreases except when $p(t)=0$. Hence, \hchange{according to La Salle's theorem, \citep[see, for example,][Ch.~5.4]{Sastry}}, $q(t)$ necessarily converges to a critical point of $f$, whereas $p(t)$ converges to zero. The origin is the only critical point contained in $\mathcal{A}_f$, which implies asymptotic stability of the origin.
\end{proof}

An eigenvalue analysis of the linearized dynamics (about the equilibrium) reveals that the eigenvalues are given by 
\begin{equation}
-d-\frac{\beta h}{2} \pm \sqrt{\left(d+\frac{\beta h}{2}\right)^2-h}, \label{eq:eigTI}
\end{equation}
where $h$ is any eigenvalue of $\diff^2 f/\diff x^2|_{x=0}$. Thus, Proposition~\ref{Prop:main} asserts that the convergence rate for all initial conditions in $\mathcal{A}_f$ is directly determined by the real part of the eigenvalues, provided that $d>0$ and $0\leq \beta < 2d/C_f$.

The eigenvalues $h$ satisfy the upper and lower bounds $1/\kappa \leq h/L \leq 1$, cf.\ \eqref{eq:defmul}. An appropriate normalization of the constants $d$ and $\beta$ reduces the analysis to the case $L=1$, which we consider in the following. Figure~\ref{Fig:EigenvalueChange2} shows how the eigenvalues vary as a function of $d, \beta$ and $h$, according to the formula \eqref{eq:eigTI}.

\begin{figure}
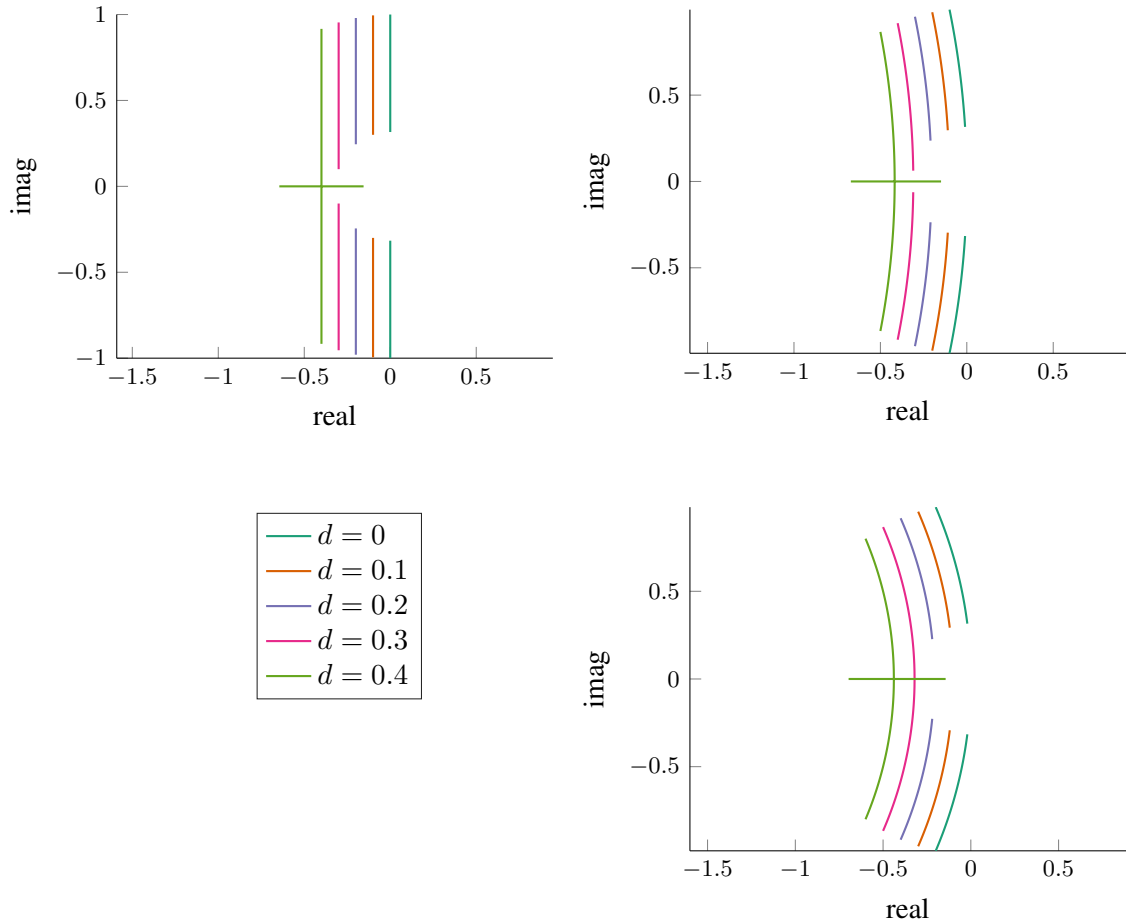

\newlength{\figurewidth}
\newlength{\figureheight}
\setlength{\figurewidth}{.4\columnwidth}
\setlength{\figureheight}{.3\columnwidth}
\begin{minipage}{.5\columnwidth}
\vspace{-3.2cm}
\input{media/PolesB0CT.tikz}
\end{minipage}%
\begin{minipage}{.5\columnwidth}
\input{media/PolesB02CT.tikz}\\ \vspace{0.5cm}\\
\input{media/PolesB04CT.tikz}
\end{minipage}
\caption{This figure shows how the eigenvalues of the linearization change as a function of the damping parameters $d$ and $\beta$, and as a function of the curvature at the equilibrium. The top left plot shows the behavior of the eigenvalues for $\beta=0$, as the curvature $h$ is varied from $0$ to $1$ (the different colors represent the different values of $d$). The top right plot shows the behavior of the eigenvalues for $\beta=0.2$ and the bottom right plot the behavior for $\beta=0.4$, where the curvature $h$ is varied from $0$ to $1$. The figure indicates that for $\beta=0$, the eigenvalues are real for large values of $d$ and small values of $h$. As $h$ is increased the eigenvalues may become complex conjugated, in which case a further increase in $h$ affects only the imaginary part. The additional parameter $\beta$ has the effect of reducing the real part for larger values of $h$. This increases the convergence rate for large values of $h$.}
\label{Fig:EigenvalueChange2}
\end{figure}

We consider first the case $\beta=0$: For small $d$---i.e., $0<d\leq 1/\sqrt{\kappa}$---the eigenvalues are complex conjugates, have real part $-d$ (independent of $h$), and their imaginary part varies between $\sqrt{1/\kappa-d^2}$ and $\sqrt{1-d^2}$ (as $h$ changes). As $d$ is increased above $1/\sqrt{\kappa}$ the eigenvalues can be both real or complex conjugates, depending on $h$. However, their worst-case real part is given by $-d+\sqrt{d^2-1/\kappa}$, which rapidly increases for larger $d$. The worst-case convergence rate (for a fixed $d$) is given by the maximum real part as $h$ is varied between $1/\kappa$ and $1$. Thus, in that sense, the optimal worst-case convergence rate is $1/\sqrt{\kappa}$ for $d=1/\sqrt{\kappa}$.

Increasing $\beta$ has the effect of increasing the convergence rate for larger values of $h$, since the parameter $\beta$ introduces additional damping as $h$ becomes large. Provided that $\beta\leq 1$, the qualitative behavior of the eigenvalues remains the same: For small $d$, i.e., $0<d \leq 1/\sqrt{\kappa} - \beta/(2\kappa)$, the eigenvalues are complex conjugates with worst-case convergence rate $-d-\beta/(2\kappa)$. As $d$ is increased above $1/\sqrt{\kappa} + \beta/(2\kappa)$, the eigenvalues can be both real or complex conjugated, depending on $h$. Their worst-case real part is again achieved for $h=1/\kappa$ and is rapidly increasing for larger $d$. Again, the optimal worst-case convergence rate is $1/\sqrt{\kappa}$ obtained for $d=-\beta/(2\kappa) + 1/\sqrt{\kappa}$.

In continuous time, any desired convergence rate can be realized, by a simple reparametrization of time. A linear reparameterization, $\hat{t}=c_\text{t} t$, where $c_\text{t}>0$ is constant, will simply scale the real parts and imaginary parts of the eigenvalues with $c_\text{t}$. A general, nonlinear, but diffeomorphic transformation will lead to time-varying dynamics, which will be discussed in Section~\ref{Sec:TV}. Thus, it might seem that any discussion of continuous-time convergence rates in the context of optimization is pointless. However, the analysis above tells us something different; it reveals how the convergence rate is affected by the condition number, which characterizes the shape of a local minimum. The analysis should be interpreted in the following way: provided that the time scale is fixed such that a convergence rate of $1$ (or $1 s^{-1}$ if the physical units are kept) is achieved for $\kappa=1$, the analysis reveals how the convergence rate of any optimization algorithm of the type \eqref{eq:exCT} deteriorates as $\kappa$ increases.

For the following analysis we introduce the notation $u_1 \succ u_2$ if the function $u_1$ dominates the function $u_2$ for large arguments; that is, $\lim_{\kappa \rightarrow \infty} u_1(\kappa)/u_2(\kappa) = \infty$, where $u_1$ and $u_2$ are real-valued functions that are positive for large arguments. In the same way, the notation $u_1 \prec u_2$ implies that the function $u_2$ dominates $u_1$ for large arguments. We further use $u_1 \sim u_2$ to imply that neither $u_1$ nor $u_2$ are dominant for large arguments; that is,
\begin{equation}
\lim_{\kappa \rightarrow \infty} \frac{u_1(\kappa)}{u_2(\kappa)} < \infty, \quad \lim_{\kappa\rightarrow \infty} \frac{u_2(\kappa)}{u_1(\kappa)} < \infty.
\end{equation}

\hchangeII{We will analyze the performance of the algorithm given in Example 1 not only on a single function, but on a whole class of functions. In order to make the following statements precise, we will fix a compact set $A\subset \mathbb{R}^{2n}$ that contains the origin and introduce the following class of functions (parametrized by the constants $\bar{\kappa}\geq 1, C_\text{f}>0, \bar{C}_\text{f}>0$).
\begin{definition}
Let $\mathcal{F}_{\bar{\kappa},C_\text{f},\bar{C}_\text{f}}$ denote the set of all functions such that each element $f\in \mathcal{F}_{\bar{\kappa},C_\text{f},\bar{C}_\text{f}}$ satisfies the following conditions: 1) $f$ has an isolated local minimum at the origin with condition number $\kappa\leq \bar{\kappa}$, 2) $f$ satisfies Assumption~\ref{Ass:fun} and the bounds in \eqref{eq:defCf}, and 3) $\mathcal{A}_f \supset A$, where $\mathcal{A}_f$ is defined according to Definition~\ref{def:SetA}. 
\end{definition} 
These conditions are motivated as follows. The first prescribes the local geometry about the local minimum at the origin, which influences the convergence rate in significant ways. The second ensures smoothness of the gradient, and the third guarantees that there are no critical points in $A$ other than the origin.\footnote{\hchangeII{It is important to exclude additional critical points from the set $A$, since these are equilibria (a gradient-based algorithm cannot distinguish between saddle points, maxima, and minima). Clearly, a bound such as \eqref{eq:boundtmp} cannot be satisfied for all $z_0\in A$ if $A$ contains equilibria other than the origin.}}}

Acceleration is obtained whenever the convergence rate (again, relative to the convergence rate achieved for $\kappa=1$) scales with $1/\sqrt{\kappa}$ for large values of $\kappa$. \hchangeII{More precisely:
\begin{definition}\label{def:Acc}
A momentum-based algorithm is accelerated if there exists constants $\kappa_0\geq 1$ and $c_\text{a}>0$, such that for any $\kappa \geq \kappa_0$ and $f\in \mathcal{F}_{\kappa,C_\text{f},\bar{C}_\text{f}}$, the following bound holds (for some constant $C_\text{a}>0$):
\begin{equation*}
    |\varphi_t(z_0)|\leq C_\text{a} |z_0| \exp(-c_\text{a} t/\sqrt{\kappa}), \quad \forall t\in I, \forall z_0 \in A.
\end{equation*}
\end{definition}
}
\hchangeII{We will now proceed to derive conditions on the parameters $d$ and $\beta$ of the algorithm given in Example 1 that guarantee accelerated convergence. We start with the case $\beta=0$.} For $d\succ 1/\sqrt{\kappa}$ it follows that the real part of the worst-case convergence rate scales with $-d+\sqrt{d^2-1/\kappa}\approx -1/(2 d \kappa)$, which makes acceleration impossible. For $d\prec 1/\sqrt{\kappa}$ the damping is too small; i.e., the real part of \eqref{eq:eigTI} scales worse than $1/\sqrt{\kappa}$. Hence, acceleration is only achieved for $d \sim 1/\sqrt{\kappa}$.
A similar argument applies to the case $0<\beta$ and yields $d+\beta/(2\kappa) \sim 1/\sqrt{\kappa}$.

The above analysis is summarized with the following proposition.
\begin{proposition}\label{Prop:AccConv}
In the nonconvex case ($C_\text{f}>0$), the algorithm given in Example 1 is accelerated for the set of parameters $d\sim 1/\sqrt{\kappa}$, $0<d$, and $0\leq \beta < 2d/C_\text{f}$.

In the convex case ($C_\text{f}=0$), the algorithm given in Example 1 is accelerated for the set of parameters $d+\beta/(2\kappa)\sim 1/\sqrt{\kappa}$, $0<d$, $0\leq \beta$.
\end{proposition}

\hchangeII{We would like to emphasize that the bound in Definition~\ref{def:Acc} is checked for each function $f\in \mathcal{F}_{\kappa,C_\text{f},\bar{C}_\text{f}}$ individually, whereby the constant $C_a$ may depend on the specific function $f$. Given the (potentially large) compact set $A$, Proposition \ref{Prop:AccConv} therefore answers the following question: What are the algorithm parameters ensuring that for any $f\in \mathcal{F}_{\kappa,C_\text{f},\bar{C}_\text{f}}$ the time for reaching an $\epsilon$ accuracy scales with $\sqrt{\kappa} \text{ln}(1/\epsilon)$ up to constants, i.e., for small $\epsilon$?} \hchange{One important aspect of the analysis is that the convergence rate depends only on the local geometry of the objective function, as characterized by the constant $\kappa$.}

\hchange{The basic notion of local analysis, obtained from a derivation of eigenvalues, is well known (see, for example, \cite{polyakBook}). But Proposition~\ref{Prop:AccConv} goes beyond classical local analysis in that, by virtue of Proposition~\ref{Prop:main}, the local rate can be \hchangeII{guaranteed for} a large portion of the region of attraction of a given equilibrium. \hchangeII{Key for this result is a nonlinear and global stability analysis.} In the strongly convex setting, the results regarding the case $\beta=0$ have been derived in \cite{sd}. Similar results regarding the case $\beta \neq 0$ can be found in \cite{Attouch}.}

\hchange{The bounds that we establish throughout the manuscript are restricted to initial conditions that are contained in a compact set within the region of attraction of a non-degenerate and isolated local minimum. This can be further motivated by considering the function shown in Figure \ref{Fig:energy}. Provided that the algorithm dissipates energy, we infer from the above discussion that the region of attraction of the origin (an open set) will comprise points in the state space that are arbitrarily close to the saddle $(\hat{x},0)$ and the maximum $(x_\text{max},0)$ (see Figure \ref{Fig:energy}). Saddle points or maxima are equilibria, since any gradient-based algorithm cannot distinguish between them and minima. Hence, when initializing the algorithm close enough to $(\hat{x},0)$ and $(x_\text{max},0)$, convergence of the algorithm can potentially take an arbitrarily long time (see \cite{SimonDu} for a formal analysis). In that sense, the restriction of initial conditions to a compact set within the region of attraction of a given equilibrium appears to be natural.}

\subsection{Example 2}\label{Sec:OptiEx2}
In discrete time, the stability analysis is not as straightforward, since the energy $H$ is in general not dissipated along trajectories. However, the specific structure of the discretization can be exploited for obtaining a nonlinear stability analysis that is valid beyond a neighborhood of a local minimum.

As remarked in \cite{ourWork}, the discretization \eqref{eq:exDT} can be divided into two parts, an energy dissipation step and a symplectic Euler step \citep[p.~3]{Hairer},
\begin{align}
\underbracket{\bar{z}_k = \Phi_{d,T}(z_k) = \left( \begin{array}{c} q_k \\ p_k - T f_\text{d}(q_k,p_k) \end{array} \right)}_{\text{energy dissipation}}, \quad \underbracket{z_{k+1} = \Phi_T(\bar{z}_k)= \left( \begin{array}{c} \bar{q}_k + T p_{k+1}\\
\bar{p}_k - T \nabla f(\bar{q}_k) \end{array} \right)}_{\text{symplectic Euler}}, \label{eq:sympEuler}
\end{align}
with intermediate state $\bar{z}_k=(\bar{q}_k,\bar{p}_k)$. In order to simplify notation, we introduce the maps $\Phi_{d,T}$ and $\Phi_T$ to denote the energy dissipation and the symplectic Euler step, respectively. The symplectic Euler step is a well-known structure-preserving first-order integration scheme. Due to the symplectic integration, there exists a modified energy function that is nearly (up to exponentially small terms) conserved by the symplectic Euler step \citep[Chapter~VI]{Hairer}. The modified energy function can be computed by means of truncated Taylor-series expansions. In order to make the analysis rigorous, $f$ is assumed to be analytic, which enables the estimation of higher-order derivatives using Cauchy's integral formula. It will be shown that for the subsequent stability analysis the assumption of $f$ being analytic poses essentially no restriction, as any continuous function can be approximated arbitrarily closely by an analytic function on a compact domain \citep[Stone-Weierstrass Theorem;][p.~159]{RudinAnalysis}. The modified energy function is characterized by the following result.

\begin{proposition}\label{Prop:ModEnergyShort}
Let $f$ be analytic on $B_r^\text{c}$, the closed $n$-dimensional ball of radius $r$ centered at the origin, and let $L_\text{H}$ be a Lipschitz constant of $\nabla H$ on  $B_r^\text{c}\times B_r^\text{c}$. Then there exists a perturbed Hamiltonian $\tilde{H}: B_r^\text{c}\times B_r^\text{c} \rightarrow \mathbb{R}$ such that
\begin{equation}
|\tilde{H}(z_0)-\tilde{H}(\Phi_T(z_0))|\leq T C_{\Delta \tilde{H}} |\nabla H(z_0)|^2 e^{-T_0/T}, \label{eq:boundH3Short}
\end{equation}
for all $0<T\leq T_0/3$ and for all $|z_0|\leq \frac{r}{2} (1+3.63 L_\text{H} T (1+ e T_0/3))^{-1}$, where 
\begin{equation*}
C_{\Delta \tilde{H}} := e(2.9+0.1 T_0) (1+e T_0/3), \quad T_0:=\frac{2\text{ln}(2)-1}{2eL_\text{H}}
\end{equation*}
are constant. The perturbed Hamiltonian has the form
\begin{equation*}
\tilde{H}(q,p)=H(q,p)-\frac{T}{2} \nabla f(q)\T p + T^2 F(q,p),
\end{equation*}
where $F:B_r^\text{c}\times B_r^\text{c} \rightarrow \mathbb{R}$ is an analytic function. The perturbed Hamiltonian $\tilde{H}$ has the same critical points as $H$ and satisfies
\begin{align}
|H(z)-\tilde{H}(z)|&\leq 15 L_\text{H} T |H(z)|, \quad \forall z\in \mathcal{A}_f \cap (B_{r/2}^\text{c}\times B_{r/2}^\text{c}), \nonumber\\
|\nabla F(z)| &\leq 356 L_\text{H}^2 |\nabla H(z)|, \quad \forall z\in B_{r/2}^\text{c}\times B_{r/2}^\text{c}. \label{eq:BoundTildeH2Short}
\end{align}
\end{proposition}

\noindent\emph{Remarks:}
\begin{itemize}
\item A more general version of Proposition~\ref{Prop:ModEnergyShort} is stated and proved in Appendix~\ref{App:ProofTildeH}.
\item The proof follows the reasoning of \citet[p.~307]{Hairer}, which enables the construction of $\tilde{H}$ by means of a truncated series expansion. Cauchy's integral formula is used to assert the convergence of the truncated series. The result from \citet[p.~307]{Hairer} is extended by exploiting the specific structure of the underlying dynamics, which leads to the additional statements about the modified energy function $\tilde{H}$. These are crucial in the context of stability analysis.
\item The constants $T_0$ and $C_{\Delta \tilde{H}}$ are determined as a function of the Lipschitz constant of $\nabla H$, which can be regarded as the natural time constant for the dynamics governed by the Hamiltonian $H$. For $L_\text{H}=1$ we obtain the following values: $T_0\approx 0.071$, $C_{\Delta \tilde{H}}\approx 8.4$,
\begin{equation}
|\tilde{H}(z_0)-\tilde{H}(\Phi_T(z_0))|\leq 8.4 T |\nabla H(z_0)|^2 e^{-0.071/T}, \label{eq:boundPert}
\end{equation}
for all $0<T\leq 0.023$ and all $z_0$ such that $|z_0|\leq 0.45 r$. Choosing, for example, a time step $T=0.001$ leads to the bound $1.2\cdot 10^{-33}$ on the right-hand side of expression \eqref{eq:boundPert}, indicating that $\tilde{H}$ is virtually exactly conserved by the symplectic Euler scheme.
\item The estimate for the maximum time step $T_0/3$ is typically conservative. However, the proposition rigorously establishes that for a small enough time step, the perturbed Hamiltonian will be almost exactly conserved. The importance lies in the fact that the upper bound on the time step is only dependent on the Lipschitz constant of $\nabla H$, which is directly related to the Lipschitz constant of $\nabla f$. Due to the fact that Proposition~\ref{Prop:ModEnergyShort} is a statement about the integration of the conservative part of the dynamics, the maximum time step $T_0/3$ is independent of the parameters $d$, $\beta$, and $\kappa$.
\item The bounds \eqref{eq:BoundTildeH2Short} enable a nonlinear stability analysis, based on the modified Hamiltonian $\tilde{H}$. Due to the fact that $15 L_\text{H} T < 1$ for all $T\leq T_0/3$, $\tilde{H}$ is necessarily positive in a neighborhood of the origin, cf.\ \eqref{eq:BoundTildeH2Short}. Combined with the fact that the perturbed Hamiltonian has the same critical points as $H$, this concludes that the level sets of $\tilde{H}$ are compact in a region about the origin. The size of this region is determined by the critical points of $H$, as argued in the analysis of Example 1; see Section~\ref{Sec:OptiEx1}.
\end{itemize}

In the following, we will use the modified energy function $\tilde{H}$, whose existence is ensured by Proposition~\ref{Prop:ModEnergyShort}, for analyzing the stability of the dynamics \eqref{eq:exDT} in the large. As pointed out, the dynamics \eqref{eq:exDT} can be subdivided into a dissipation step and a symplectic Euler step.

In order to simplify the presentation we focus on the case where $T$ is small and the map $\Phi_{d,T}$ is close to the identity (little damping). For stability analysis this is the most challenging setting, since, due to the almost vanishing damping, the convergence can be arbitrarily slow (the equilibrium is almost non-attractive). In case $\Phi_{d,T}$ is not close to the identity, which is, for example, obtained for a constant parameter $\beta>0$, independent of $\kappa$, stability can be analyzed by means of the unperturbed energy function $H$, as shown in Appendix~\ref{App:DisStabilityConvex}. We will thus concentrate on the following result.

\begin{proposition}\label{Prop:StabilityDT}
Let the \hchange{dissipative forces} $f_\text{d}(q_k,p_k)$ be such that $-d_2 |p_k|^2 \leq p_k\T f_\text{d}(q_k,p_k)\leq -d_1 |p_k|^2$ with $0 < d_1 \leq d_2$.

Then, there exists a maximum time step $T_\text{max}>0$, such that the origin is an asymptotically stable equilibrium of the dynamics \eqref{eq:exDT} for all $T\leq T_{\text{max}}$, with domain of attraction at least $\mathcal{A}_f\cap A$, where $A$ is any compact set. Up to exponentially small terms due to \eqref{eq:boundH3Short}, the maximum time step $T_\text{max}$ depends on an upper bound on $d_2$, $d_1/d_2$, and $L_\text{H}$, the Lipschitz constant of $\nabla H$.
\end{proposition}

\noindent\emph{Remarks:} 
\begin{itemize}
\item The proof of Proposition~\ref{Prop:StabilityDT}, which is included in Appendix~\ref{App:ProofDTStab}, is based on the Lyapunov function
\begin{equation*}
V(q,p)=\tilde{H}(q,p) + \frac{T d_1}{2} \nabla f (q)\T p,
\end{equation*}
where $\tilde{H}$ denotes the perturbed energy function introduced in Proposition~\ref{Prop:ModEnergyShort}. The function $V$ is motivated by the following observation. As the damping parameter $d_1$ decreases and $\Phi_{d,T}$ approaches the identity, $V$ reduces to the perturbed energy function $\tilde{H}$, which (up to exponentially small terms) is known to be conserved by $\Phi_{T}$. The correction $T d_1 p\T \nabla f(q)/2$ is required due to the fact that $\Phi_{d,T}$ is a pure contraction in the momentum variable, and as a result $\tilde{H}$ is not necessarily decreasing through the application of $\Phi_{d,T}$. \hchange{A similar correction has been previously used in the literature on dissipative systems; see, for example, \cite{hale} or \cite{Haraux}, and in the context of the heavy-ball method in \cite{sd}.} 
\item Proposition~\ref{Prop:StabilityDT} ensures that the region of attraction is the same as the continuous-time counterpart, and that the requirements on the time step for guaranteeing asymptotic stability is independent of the minimum damping $d_1$ (up to the exponentially small terms). This will be important in the following, where we will analyze how the convergence rate scales with $\kappa$.
\item Provided that $\beta>0$ and the function $f$ is convex, the stability analysis simplifies substantially. In particular, a sufficient condition for stability is obtained by analyzing the total energy $H$ along the trajectories of \eqref{eq:exDT}, as shown in Appendix~\ref{App:DisStabilityConvex}.
\end{itemize}

%

Next we will analyze the specific implications for the dynamics \eqref{eq:exDT}. We recall that the contraction step $\Phi_{d,T}$ is given by 
\begin{equation*}
p_{k+1}=p_{k} + T f_{\text{d}}(q_k,p_k) = (1-2dT) p_k - T \int_{0}^{\beta} \left.\frac{\diff^2 f}{\diff x^2}\right|_{q_k+\tau p_k} \diff \tau ~p_k,
\end{equation*}
which concludes that the upper and lower bounds $d_2$ and $d_1$ of Proposition~\ref{Prop:StabilityDT} are given by $d_2=2d+ \bar{C}_\text{f} \beta$ and $d_1=2d-C_\text{f} \beta$ (assuming $\beta\geq 0$). The upper and lower bounds $d_2$ and $d_1$ are therefore functions of $\kappa$.

Hence, according to Proposition~\ref{Prop:StabilityDT}, provided that $d_1>0$ and $d_2/d_1$ and $d_2$ are bounded with respect to $\kappa$, there exists a time step $T$, \emph{independent} of $\kappa$, such that the origin is asymptotically stable, where the region of attraction includes any compact subset of $\mathcal{A}_f$. Then, according to Proposition~\ref{Prop:main}, the convergence rate of any trajectory starting in $A\subset \mathcal{A}_f$, $A$ compact, is determined by the magnitude of the eigenvalues of the linearized dynamics. The eigenvalues are given by
\begin{equation*}
\lambda_{1,2}=1- T \left( d + \frac{\beta h}{2} + \frac{Th}{2} \pm \sqrt{(d+\frac{\beta h}{2} + \frac{Th}{2})^2 -h}\right),
\end{equation*}
enabling the calculation of the convergence rate for a given choice of $T$, $d$, $\beta$ and $h$.

For the following discussion, we again assume that $d$ and $\beta$ are normalized such that $1/\kappa \leq h \leq 1$. We fix $d$, $\beta$, and $T$ and analyze how the eigenvalues change as a function of $h$, as done in Figure~\ref{Fig:EigenvalueChange}. The worst-case convergence rate is determined by the maximum magnitude of the eigenvalues (over $1/\kappa \leq h \leq 1$). For very small values of $h$, the eigenvalues are real, and one eigenvalue is very close to $1$. As $h$ is increased, the eigenvalues become complex conjugates, where for $\beta=0$, the eigenvalues are located along circles centered at the origin, whereas for $\beta>0$ their magnitude slightly decreases. If $h$ is increased further, the eigenvalues become real again and their magnitude increases. The worst-case convergence rate, i.e., the largest magnitude of $|\lambda_{1,2}|$ is therefore either achieved for $h=1/\kappa$ or $h=1$.

We are interested in determining the conditions on $d$ and $\beta$ such that the convergence rate scales with $1/\sqrt{\kappa}$. We consider first the case $h=1/\kappa$ and $\beta=0$, and assume that $d$ is large enough such that the eigenvalues are real. Then, provided that $d\sim 1/\sqrt{\kappa}$ it follows that $\lambda_{1,2} \sim 1-T/\sqrt{\kappa}$. However, if $d$ is chosen to be larger, i.e., $d \succ 1/\sqrt{\kappa}$, it follows that for large $\kappa$, 
\begin{equation}
|\lambda_{1,2}|\approx 1-T d  + T d \sqrt{1 - \frac{1/\kappa}{d^2}}\approx 1-\frac{T}{2d\kappa},
\end{equation}
that is, the convergence rate scales worse than $1/\sqrt{\kappa}$. In case $d$ is small enough, such that the eigenvalues are complex conjugates, their magnitude is given by
\begin{equation}
|\lambda_{1,2}|=\sqrt{1-2dT},\label{eq:form1}
\end{equation}
which indicates that a scaling of the convergence rate with $1/\sqrt{\kappa}$ is only achieved for $d \sim 1/\sqrt{\kappa}$.

For $h=1$ and $\beta=0$ it follows from $d\sim 1/\sqrt{\kappa}$ that $\lambda_{1,2}$ approach
\begin{equation}
\lambda_{1,2} \rightarrow 1-\frac{T^2}{2} \mp T \sqrt{\frac{T^2}{4} -1},
\end{equation}
for large $\kappa$. Thus, for $T\leq 1$, for example, the eigenvalues are complex conjugates and the worst-case convergence rate scales with $1/\sqrt{\kappa}$. The condition $d\sim 1/\sqrt{\kappa}$ is therefore necessary and sufficient for ensuring that the worst-case magnitude of $|\lambda_{1,2}|$ scales with $1/\sqrt{\kappa}$. A similar analysis applies to the case where $\beta > 0$, as shown below.


As in the discussion of Section~\ref{Sec:OptiEx1} we say that the optimization algorithm \eqref{eq:exDT} is accelerated if the convergence rate scales with $1/\sqrt{\kappa}$ \hchangeII{(see Definition~\ref{def:Acc})}, where exponentially small terms resulting from the application of Proposition~\ref{Prop:ModEnergyShort} are neglected. As a consequence, the above discussion allows us to translate the results from Proposition~\ref{Prop:AccConv} almost verbatim to the discrete-time setting. This results in a broad characterization of the parameters $d$ and $\beta$ leading to acceleration.

\begin{proposition}\label{Prop:AccDis}
In the nonconvex case, ($C_f>0$), there exists a time step $T>0$, such that the algorithm given in Example 2 is accelerated provided that $d\sim 1/\sqrt{\kappa}$, $0<d$, $0\leq \beta<2d/C_f$.

In the convex case, ($C_f=0$), there exists a time step $T>0$, such that the algorithm given in Example 2 is accelerated provided that $d+ \beta/(2\kappa) \sim 1/\sqrt{\kappa}$, $0<d$, $0\leq \beta$ and either $\beta\sim 1$, $\beta \sim 1/\sqrt{\kappa}$, or $\beta \prec 1/\sqrt{\kappa}$.
\end{proposition}
\begin{proof}
The requirements on $\beta$ are needed for guaranteeing asymptotic stability of the origin for a time step $T$ that is small enough, but independent of $\kappa$, as implied by Proposition~\ref{Prop:StabilityDT} and Appendix~\ref{App:DisStabilityConvex}. It therefore remains to analyze the behavior of the eigenvalues $\lambda_{1,2}$ as a function of $h$ for $1/\kappa \leq h \leq 1$.

For small $h$ the eigenvalues are real. As $h$ is increased, they may become complex conjugates, and if $h$ is increased further they become real again. In case the eigenvalues are complex conjugates, their magnitude is given by
\begin{equation}
|\lambda_{1,2}|=\sqrt{1-2dT-\beta hT}.\label{eq:form11}
\end{equation}

We consider first the case $h=1/\kappa$: The eigenvalues are real provided that $d+\beta/(2\kappa) + T/(2\kappa) \geq 1/\sqrt{\kappa}$, which, given the assumptions on $\beta$, implies that $d \geq 1/\sqrt{\kappa}$ for large $\kappa$. In case $d \sim 1/\sqrt{\kappa}$ it follows that $\lambda_{1,2}\sim 1-T/\sqrt{\kappa}$, which therefore yields accelerated convergence. In case $d\succ 1/\sqrt{\kappa}$ it follows that
\begin{equation}
(d+\frac{\beta}{2\kappa} + \frac{T}{2\kappa}) \sqrt{1 - \frac{1/\kappa}{(d+\frac{\beta}{2\kappa} + \frac{T}{2\kappa})^2}}\approx (d+\frac{\beta}{2\kappa} + \frac{T}{2\kappa})-\frac{1}{2d\kappa + \beta + T},
\end{equation}
due to the fact that $1/(d^2 \kappa)\rightarrow 0$ for large $\kappa$. Thus, accelerated convergence is impossible for $d \succ 1/\sqrt{\kappa}$. 
In case the eigenvalues are complex conjugates for $h=1/\kappa$, it follows from \eqref{eq:form11} that $d\sim 1/\sqrt{\kappa}$ is necessary for accelerated convergence.

We therefore proceed to the case $h=1$, where due to the fact that $d$ vanishes for large $\kappa$, the two eigenvalues approach constant values,
\begin{equation}
\lambda_{1,2} \rightarrow  1- T \left(\frac{\beta^\infty}{2} + \frac{T}{2} \pm \sqrt{(\frac{\beta^\infty}{2} + \frac{T}{2})^2 - 1}\right),
\end{equation}
for $\beta^\infty:=\lim_{\kappa\rightarrow \infty} \beta$, which is either constant or zero. This results in a convergence rate that is independent of $\kappa$, since $T$ is chosen small enough to guarantee convergence.
\end{proof}

In particular, this leads to the conclusion that, for example, the following heavy-ball scheme,
\begin{equation}
q_{k+1}=q_k + T p_{k+1}\qquad 
p_{k+1}=p_k - (2T/\sqrt{\kappa}) p_k - T \nabla f(q_k)\label{eq:HB}
\end{equation}
leads to accelerated convergence for a sufficiently small value of $T$. In case the Lipschitz constant of $\nabla f$ is bounded by one, the bound $T\leq 0.001$ for guaranteeing acceleration is obtained by following the proof of Proposition~\ref{Prop:StabilityDT}.\footnote{Strictly speaking, due to the exponential terms in \eqref{eq:boundH3Short}, the bound $T\leq 0.001$ is valid for $\kappa \leq 4.1\cdot  10^{52}$.} Note that compared to Nesterov's original scheme, a gradient evaluation at $q_k +\beta p_k$ is not required. 

\hchange{We would like to emphasize that our analysis captures the computational complexity up to constants. For a given objective function, choosing $\beta \neq 0$ typically allows larger steps $T$, which is likely to reduce the number of iterations needed for convergence. On strongly convex functions, for example, the ``Constant step scheme III'' presented in \citet[p. 81]{NesterovIntro}, seems to provide a good compromise between a straightforward implementation and fast convergence (considering constants).}

\hchange{Local results akin to Proposition \ref{Prop:AccDis} which are based on a quadratic objective function are well known. A summary can be found for example in \cite{LessardRecht}. In addition to unifying the continuous-time and discrete-time analysis, an important aspect of Proposition \ref{Prop:AccDis} is to establish a rate that is valid for a large portion of the region of attraction of a given equilibrium. Furthermore, the fact that the heavy-ball scheme \label{eq:HB} leads to accelerated convergence for a sufficiently small value of $T$ appears to be new. If $T$ is chosen to be too large, however, for example by tuning $T$ on quadratic functions, this may lead to convergence issues even on strongly convex functions, as has been reported in \cite{LessardRecht}.}

\begin{figure}
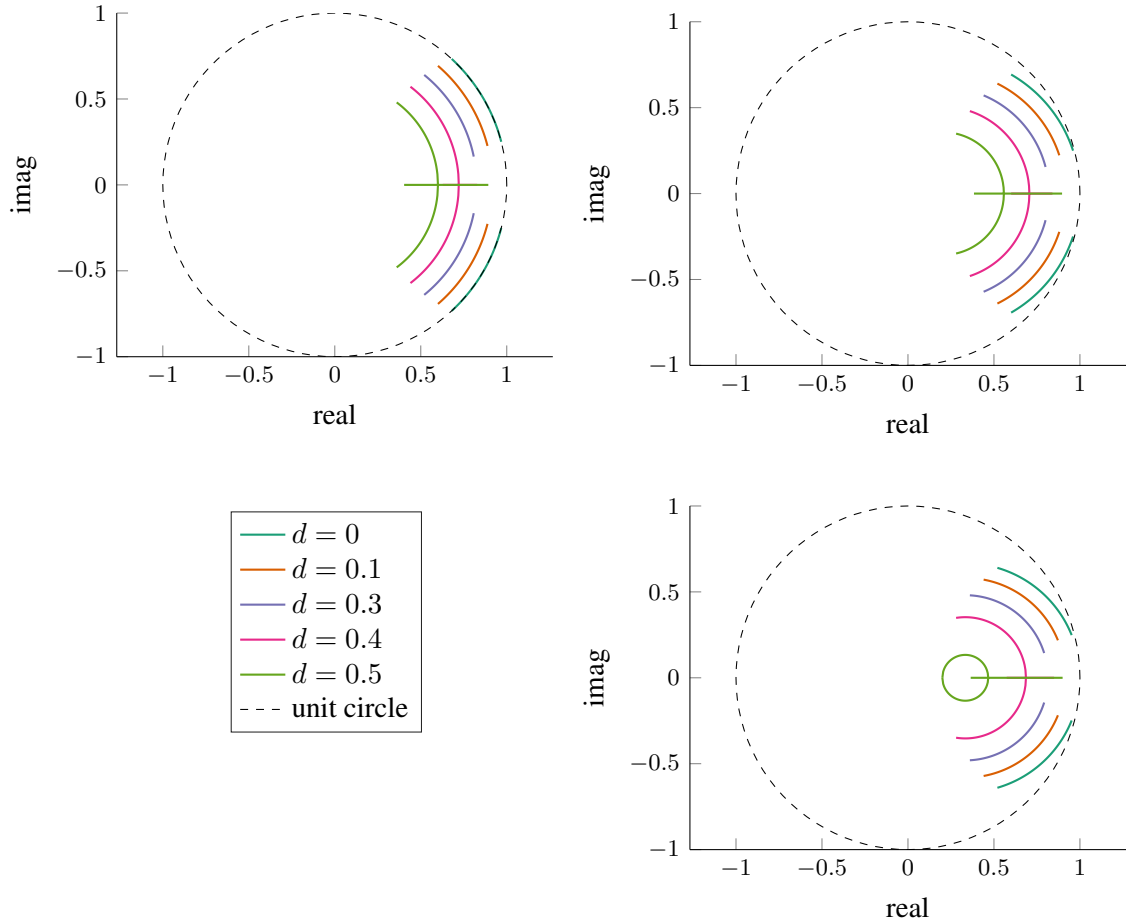

\setlength{\figurewidth}{.4\columnwidth}
\setlength{\figureheight}{.3\columnwidth}
\begin{minipage}{.5\columnwidth}
\vspace{-2.7cm}
\input{media/PolesB0.tikz}
\end{minipage}%
\begin{minipage}{.5\columnwidth}
\input{media/PolesB02.tikz}\\ \vspace{0.1cm}\\
\input{media/PolesB04.tikz}
\end{minipage}
\caption{This figure shows how the eigenvalues of the linearization change as a function of the damping parameters $d$ and $\beta$, and as a function of the curvature at the equilibrium, $h$. The top left plot shows the behavior of the eigenvalues for $\beta=0$, as the curvature $h$ is varied from $0.1$ to $1$ (the different colors represent the different values of $d$). The top right plot shows the behavior of the eigenvalues for $\beta=0.2$ and the bottom right plot the behavior for $\beta=0.4$, as the curvature $h$ is varied from $0.1$ to $1$. For all the plots the time step is set to $T=0.8$. The figure indicates that for $\beta=0$, the eigenvalues move first along the real axis and then along concentric circles (as $h$ changes). The additional parameter $\beta$ has the effect of reducing the magnitude of the eigenvalues for larger values of $h$. While this can increase the convergence rate, it bears also the risk of instability, when $d$, $\beta$, and $T$ are chosen to be too large.}
\label{Fig:EigenvalueChange}

\end{figure}



\section{The Time-Varying Case}\label{Sec:TV}
Next we will consider the case where the dynamics $g$ are time dependent. The main motivation for introducing time-dependent dynamics lies in improving the convergence rate for functions that have an almost vanishing curvature at a local minimum.

We slightly abuse notation and reuse the variables introduced in Section~\ref{Sec:CharTI}. It will be clear from context whether we refer to the time-varying or time-invariant case.

Following Section~\ref{Sec:TI}, we model a momentum-based optimization algorithm as a continuous-time or discrete-time dynamical system of the form
\begin{gather}
q^+(t)=g_\text{q}(t,q(t),p(t)), \quad 
p^+(t)=g_\text{p}(t,q(t),p(t)), \quad \forall t\in I, t\geq t_0, \label{eq:mombased2TV}\\
q(t_0)=q_0, \quad p(t_0)=p_0, \quad t_0\in I, \label{eq:mombased2TVi}
\end{gather}
where the dynamics satisfy the following assumptions:\footnote{The assumptions can be relaxed to mere Lipschitz continuity of the second and third argument as discussed in Appendix~\ref{App:Smooth}. }
\begin{assumption}
The dynamics $g_\text{q}$ and $g_\text{p}$ are continuously differentiable in all their arguments. The derivative in the second and third arguments is Lipschitz continuous, uniformly in $t$.
%
\end{assumption}

We define the map from $(q_0,p_0,t_0) \rightarrow (q(t),p(t))$ as $\varphi_t: \mathbb{R}^{2n} \times I \rightarrow \mathbb{R}^{2n}$, for $t\in I$, $t\geq t_0$, and define the dynamical system \eqref{eq:mombased2TV}-\eqref{eq:mombased2TVi} to be a \emph{momentum-based optimization algorithm} for the function $f$ if $x^*=0$ is an asymptotically stable equilibrium in the sense of Lyapunov (uniformly in the initial time $t_0$). Due to the continuity assumptions on the dynamics $g_\text{q}$ and $g_\text{p}$, the continuous-time existence and uniqueness results, as well as the continuity of trajectories with respect to their initial conditions continues to hold \citep[p.~93, Corollaries~3~and~4]{DifferentialEquationsArnold}.

The examples presented in Section~\ref{Sec:Ex1} and Section~\ref{Sec:Ex2} are modified by allowing the constants $d>0$ and $\beta\geq 0$ to be time varying. We restrict ourselves to the case where $d$ and $\beta$ converge for large time, which leads to asymptotically time-invariant dynamics. This has the advantage that the asymptotic convergence rate is independent of the initial time $t_0$, and, under mild continuity conditions, the (optimal) convergence rates from Section~\ref{Sec:CharTI} will be recovered for large $t$. Thus, \eqref{eq:exCT} and \eqref{eq:exDT} are extended by replacing the constants $d$ and $\beta$ with the functions $d: \mathbb{R}_{\geq 0} \rightarrow \mathbb{R}_{>0}$ and $\beta: \mathbb{R}_{\geq 0} \rightarrow \mathbb{R}_{\geq0}$, which are assumed to be continuously differentiable and convergent,
\begin{equation}
d_\infty:=\lim_{t\rightarrow \infty} d(t)>0, \quad \beta_\infty:=\lim_{t\rightarrow \infty} \beta(t)\geq 0.
\end{equation}

The stability analysis of Example~1 (Section~\ref{Sec:Ex1}) translates to the time-varying case.\footnote{La Salle's invariance principle must be extended \citep[as discussed in, e.g.,][p.~205]{Sastry}.} This implies that the algorithm \eqref{eq:exCT} with time-varying parameters $d$ and $\beta$ is a momentum-based optimization algorithm according to the above definition. 

The stability analysis of Example~2 is extended to the time-varying case by rewriting the dynamics \eqref{eq:mombased2TV} in the following way:
\begin{equation}
z_{t+1}=\left(\underbrace{\left.\frac{\partial g}{\partial z}\right|_{t\rightarrow\infty,0}}_{\text{linear time-invariant part}} + \underbrace{\left( \left.\frac{\partial g}{\partial z}\right|_{t,0} - \left.\frac{\partial g}{\partial z}\right|_{t\rightarrow\infty,0} \right)}_{\text{linear time-variant part}} \right) z_t + \underbrace{r_\text{NL}(t,z_t)}_{\text{nonlinear part}}, \label{eq:dynamicSysDTtmp}
\end{equation}
where the linear time-varying part vanishes for $t\rightarrow \infty$. The remainder $r_\text{NL}$ is of second order in $z_t$ (uniformly in $t\in I$). Thus, if the linear time-invariant part has all eigenvalues strictly within the unit circle, the dynamics \eqref{eq:dynamicSysDTtmp} are asymptotically stable, uniformly in $t_0$.\footnote{The linear part then converges at an exponential rate, as can be shown by Lemma~\ref{Lem:GronwallTV} (see Appendix~\ref{App:PropMainTV}). An argument similar to \citet[Thm.~1']{Bellman} can be applied to argue that the nonlinear dynamics are uniformly asymptotically stable in a neighborhood of the origin.} Consequently, the analysis from Section~\ref{Sec:Ex2} yields the following conditions, cf.\ \eqref{eq:DTlinCond}:
\begin{equation}
0< T(2d_\infty+\beta_\infty h) \leq 2- h T^2/2, \quad d_\infty>0, \quad \beta_\infty\geq 0,
\end{equation}
for all eigenvalues $h$ of $H_\text{e}=\left.\diff^2 f/\diff x^2\right|_{x=0}$.

\subsection{Characterizing the convergence rate}\label{Sec:CharRateTV}
As in Section~\ref{Sec:CharTI} we argue that a linear analysis can be used to characterize the convergence rate of the nonlinear dynamics up to constants. The results are derived only for non-degenerate isolated local minima. \hchange{We believe, however, that the discussion also provides insights into the case where the curvature vanishes at a local minimum; see Section~\ref{Sec:OptiEx1TV} for further discussion of this point.}

\hchange{A similar analysis of the role of choosing a time-varying damping parameter can be found in \cite{Attouch}. The emphasis of the following section lies in providing intuition through a unified treatment between the degenerate and non-degenerate case.}

We make the following assumption.
\begin{assumption}\label{Ass:linTV}
Let the linearized dynamics
\begin{equation*}
\delta z^+(t)= \left.\frac{\partial g}{\partial z}\right|_{t,z=0} \delta z(t), \quad \forall t\in I, \quad t\geq \tau, \quad \delta z(\tau)=z_0,
\end{equation*}
be such that there is an estimate
\begin{equation*}
|\delta z(t)| \leq C_\text{l} |\delta z(\tau)| \frac{\rho(t)}{\rho(\tau)} \exp(-\alpha (t-\tau)), \quad \forall z_0\in \mathbb{R}^{2n}, \quad \forall \tau,t\in I, t\geq \tau,
\end{equation*}
where $C_\text{l}\geq 1$ and $\alpha>0$ are constant, and $\rho:\mathbb{R}_{\geq 0} \rightarrow \mathbb{R}_{\geq 0}$ is continuous, and monotonically decreasing.
\end{assumption}

\begin{proposition}\label{Prop:mainTV}
Let Assumption~\ref{Ass:linTV} be satisfied and let the region of attraction of the equilibrium at the origin of the nonlinear dynamics \eqref{eq:mombased2TV} be denoted by $\mathcal{R}\subset \mathbb{R}^{2n}$. Then, for any compact set $A \subset \mathcal{R}$ and any initial time $t_0\in I$, there exists a finite constant $\hat{C}\geq 1$ such that for all $z_0\in A$,
\begin{equation*}
|\varphi_t(z_0,t_0)|\leq \hat{C} |z_0| \frac{\rho(t)}{\rho(t_0)} \exp(-\alpha (t-t_0)), \quad \forall t\in I, \quad t\geq t_0.
\end{equation*}
\end{proposition}
\begin{proof}
The set $A$ is compact, nonempty, and contains the origin. Lemma~\ref{Lem:AsymptConvTV} (see Appendix~\ref{App:PropMainTV}) implies that there is a ball $B_\delta$ of radius $\delta > 0$, centered at the origin, such that any trajectory starting in $B_\delta$ at time $\tau \in I$ converges with rate $(\rho(t)/\rho(\tau)) \exp(-\alpha (t-\tau))$. The following claim can be verified with the arguments of Proposition~\ref{Prop:main}.

\noindent\emph{Claim:} There exists a finite time $T_\text{m}>t_0$, $T_\text{m}\in I$ such that for all $z_0\in A$, $\varphi_{T_\text{m}}(z_0,t_0) \in B_\delta$.

Moreover, the uniform continuity of $\varphi_t$ for all $t\in I$, $t_0 \leq t \leq T_\text{m}$, implies further that $\varphi_t(A,t_0)$ is bounded for any $t\in I$, $t_0\leq t \leq T_\text{m}$. Combined with Lemma~\ref{Lem:AsymptConvTV}, this yields the following bound
\begin{equation}
|\varphi_t(z_0,t_0)|\leq 
\begin{cases} C_\text{A} &t_0\leq t \leq T_\text{m}, t\in I,\\
\delta \tilde{C} (\rho(t)/\rho(T_\text{m})) \exp(-\alpha (t-T)) &t>T_\text{m}, t\in I,
\end{cases}\label{eq:prooftmp1111}
\end{equation}
for all $z_0\in A$, where $C_\text{A}\geq\delta$ and $\tilde{C}\geq 1$ are positive constants. We fix $z_0\in A$, consider the trajectory $z(t):=\varphi_t(z_0,t_0)$, $t\in I$, and apply the mean value theorem,
\begin{equation}
z^+(t)=\left(\left.\frac{\partial g}{\partial z}\right|_{t,z=0} + \left.\frac{\partial g}{\partial z}\right|_{t,z=\xi(t)} - \left. \frac{\partial g}{\partial z}\right|_{t,z=0} \right) z(t),
\end{equation}
where $\xi(t)$ lies between $z(t)$ and the origin. Due to the fact that the dynamics are assumed to have Lipschitz continuous derivatives (uniformly in $t$), we obtain the following bound 
\begin{equation}
\left\vert \left.\frac{\partial g}{\partial z} \right|_{t,z=\xi(t)} - \left.\frac{\partial g}{\partial z}\right|_{t,z=0} \right\vert \leq \bar{C}_\text{A} |z(t)|,
\end{equation}
where $\bar{C}_\text{A}$ denotes the Lipschitz constant of $\partial g/\partial z$ on $A$ (uniformly in $t$). According to \eqref{eq:prooftmp1111}, the trajectory $|z(t)|$ is integrable (in continuous time) and absolutely summable (in discrete time). We obtain, by virtue of Lemma~\ref{Lem:GronwallTV},
\begin{equation}
|z(t)|\leq C_\text{l} \exp(C_\text{l} \bar{C}_\text{A} C_\text{z}) |z_0| \frac{\rho(t)}{\rho(t_0)} \exp(-\alpha (t-t_0)),  \quad \forall z_0 \in A,
\end{equation}
where $C_\text{z}$ is constant. The constant $C_\text{z}$ is related to an upper bound on the integral (in continuous time) or the sum (in discrete time) of $|z(t)|$ over $t\in I$, which according to \eqref{eq:prooftmp1111}, is guaranteed to be finite.
\end{proof}


\subsection{Example 1}\label{Sec:OptiEx1TV}
\hchange{The energy function, as defined in \eqref{eq:energy}, is not explicitly dependent on time, even if the parameters $d$ and $\beta$ are chosen to be time-varying.} Following the discussion of Section~\ref{Sec:OptiEx1}, we conclude that $\mathcal{A}_f$ is contained in the region of attraction of the equilibrium at the origin, as long as 
\begin{gather*}
0<d(t), \quad 0\leq \beta(t)<2d(t)/C_\text{f}, \quad \forall t\in I,\\
0<d_\infty, \quad 0\leq \beta_\infty < 2 d_\infty/C_\text{f}.
\end{gather*}
As a result, Proposition~\ref{Prop:mainTV} implies that the convergence rate of the trajectories starting from $A \subset \mathcal{A}_f$, $A$ compact, is given by the linearization of \eqref{eq:exCT} about the origin.

After a change of coordinates that diagonalizes $H_\text{e}$, the linearized dynamics of a single coordinate are given by
\begin{equation}
\delta \dot{q}(t)=\delta p(t), \qquad \delta \dot{p}(t)=(-2d(t) - \beta(t) h) \delta p(t) - h \delta q(t), \label{eq:tvlinearization}
\end{equation}
where $\delta q(t)\in \mathbb{R}$, $\delta p(t)\in \mathbb{R}$, and $h$ denotes the corresponding eigenvalue of $H_\text{e}$. As in Section~\ref{Sec:OptiEx1}, upper and lower bounds on the eigenvalues of $H_\text{e}$ are given by $1/\kappa \leq h \leq 1$. The analysis simplifies by expressing the dynamics using the coordinates $\delta \tilde{q}$, which are defined by
\begin{equation}
\exp\left(-\int_{t_0}^{t} \bar{d}(\tau) \diff \tau\right) ~\delta \tilde{q}(t):= \delta q(t), \label{eq:cochange}
\end{equation}
with $\bar{d}(t):=d(t)+\beta(t) h/2$, and yields
\begin{equation}
\delta \ddot{\tilde{q}}(t)=-(-\dot{\bar{d}}(t)-\bar{d}(t)^2 + h) \delta \tilde{q}(t). \label{eq:time-varyingHarmonicOscillator}
\end{equation}
The dynamics \eqref{eq:time-varyingHarmonicOscillator} represent an undamped linear harmonic oscillator whose frequency varies with time. This time-varying frequency can be interpreted as the time-varying generalization of the square-root term in \eqref{eq:eigTI}. The non-square-root term in \eqref{eq:eigTI} is captured by the coordinate change \eqref{eq:cochange}.

In the following, the choice
\begin{equation}
\dot{d}(t) =-d(t)^2+d_\infty^2, \quad \beta(t)=0,\label{eq:Riccatti}
\end{equation}
is discussed. This leads to a time-invariant right-hand side of \eqref{eq:time-varyingHarmonicOscillator}, and simplifies the subsequent analysis. The more general case (no restriction on $\bar{d}$) can be analyzed numerically, or by approximating $\bar{d}$ with a rational function and applying tools from asymptotic analysis \citep{AsymptoticAnalysis}. The solutions of \eqref{eq:Riccatti} monotonically approach $d_\infty$ as $t$ increases, whereby the time-invariant case discussed in Section~\ref{Sec:OptiEx1} is recovered with $d(0)=d_\infty$ ($d(0)=d_\infty$ implies $d(t)=d_\infty$ for all $t\geq 0$). The coordinate transformation \eqref{eq:cochange} therefore implies that, compared to the time-invariant case, the convergence rate can be improved by choosing $d(0)>d_\infty$, as this increases the area underneath the curve $d(t)$. Figure~\ref{Fig:Area} illustrates the situation. 

\begin{figure}
\center
\begin{tikzpicture}[scale=1.5]
\pgfkeys{/pgf/declare function={arctanh(\x) = 0.5*(ln(1+\x)-ln(1-\x));}}
\pgfkeys{/pgf/declare function={dfun(\t,\kappa)=1/sqrt(\kappa) * (tanh(1/sqrt(\kappa)* \t) + sqrt(\kappa))/(1+tanh(1/sqrt(\kappa)* \t)*sqrt(\kappa));}}

\def\kappa{10}
\pgfmathsetmacro\val{1/sqrt(\kappa)};

\draw[domain=0:5, variable=\t]
plot({\t}, {dfun(\t,\kappa)});

\draw[domain=0:5, variable=\t] plot({\t}, {1/sqrt(\kappa)});

\draw [thick] [->] (0,0)--(5.5,0) node[right, below] {$t$};
\draw [thick] [->] (0,0)--(0,1.5) node[left]{$d$};

\fill [gray, domain=0:5, variable=\t]
(0,\val) 
-- plot ( {\t}, {dfun(\t,\kappa)})
-- (5,\val)
-- cycle;

\draw[yshift=\val cm, thick] (-3pt,0pt) -- (3pt,0pt) node[left=0.2cm]{$d_\infty$};
\draw[yshift=1cm, thick] (-3pt,0pt) -- (3pt,0pt) node[left=0.1cm]{$d(0)$};

\end{tikzpicture}
\caption{The figure shows $d(t)$ defined according to \eqref{eq:Riccatti} with $d(0)=1$ and $d_\infty=1/\sqrt{10}$ as an example. Compared to the time-invariant dynamics (obtained in the asymptotic limit), which converge with rate $d_\infty$, the time-varying terms improve the convergence rate by the shaded area, as indicated with \eqref{eq:cochange}.}
\label{Fig:Area}
\end{figure}
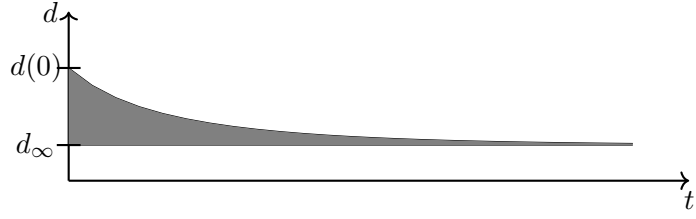

This intuition can be made quantitative by noticing that the solutions of \eqref{eq:Riccatti} are given by 
\begin{align}
d(t)&= d_\infty ~ \text{tanh}(d_\infty (t+C_\text{d})), \quad t\geq t_0,
\end{align}
where $C_\text{d}$ is constant and related to $d(t_0)$. This implies
\begin{align}
\text{exp}(-\int_{t_0}^{t} \bar{d}(\tau) \diff \tau)&=\frac{\text{cosh}(d_\infty (t_0+C_\text{d}))}{\text{cosh}(d_\infty (t+C_\text{d}))}=\frac{e^{-2 d_\infty (C_\text{d}+t_0)} + 1}{e^{-2 d_\infty (C_\text{d}+t)} + 1} e^{-d_\infty (t-t_0)}\nonumber\\
&=\frac{d_\infty+d(t)}{d_\infty + d(t_0)} e^{-d_\infty (t-t_0)}=:\frac{\rho_\text{tv}(t)}{\rho_\text{tv}(t_0)} e^{-d_\infty (t-t_0)}.
\end{align}
Due to the fact that \eqref{eq:time-varyingHarmonicOscillator} reduces to a time-invariant harmonic oscillator, the solutions to \eqref{eq:tvlinearization} can be expressed in closed form with elementary functions. The fundamental solutions to \eqref{eq:tvlinearization} are therefore given by\footnote{We assume $h\neq d_\infty^2$. The case $h=d_\infty^2$ is obtained by considering the solutions of \eqref{eq:tvlinearization} either as the limit $h\downarrow d_\infty^2$ or the limit $h\uparrow d_\infty^2$.}
\begin{align}
 \frac{\rho_\text{tv}(t)}{\rho_\text{tv}(t_0)} e^{(-d_\infty \pm \sqrt{d_\infty^2-h}) (t-t_0)}.\label{eq:fundSol}
\end{align}
The rate of the exponential decay behaves in the same way as the real part of the expression \eqref{eq:eigTI} ($d_\infty$ corresponds to $d$ in \eqref{eq:eigTI}) and as a result, the discussion for the time-invariant case applies here in the same way. This concludes that accelerated convergence occurs as long as $d_\infty \sim 1/\sqrt{\kappa}$. Compared to the time-invariant case, however, the convergence is improved by the factor $\rho_\text{tv}(t)/\rho_\text{tv}(t_0)$. For large values of $d_\infty$ the difference is not substantial (at most a constant factor of size $\lim_{t\rightarrow \infty} \rho_\text{tv}(t)/\rho_\text{tv}(t_0)=2 d_\infty/(d_\infty + d(t_0))$ can be gained). However, the improvement becomes crucial for small values of $d_\infty$, since $\rho_\text{tv}(t)/\rho_\text{tv}(t_0)$ behaves as
\begin{equation*}
\frac{\rho_\text{tv}(t)}{\rho_\text{tv}(t_0)} \approx \frac{1}{1+d(t_0) (t-t_0)},
\end{equation*}
for small $t$. 
Thus choosing $d_\infty\sim 1/\sqrt{\kappa}$, ensures convergence roughly with $1/(1+d(t_0)(t-t_0))$ for small $t$, and exponential convergence with rate 
$1/\sqrt{\kappa}$ for large $t$; the transition occurs for $(t-t_0) \approx 1/d_\infty$. The situation is shown in Figure~\ref{Fig:TVdynamics}. 

\begin{figure}
\setlength{\figurewidth}{.7\columnwidth}
\setlength{\figureheight}{.3\columnwidth}
\center
\input{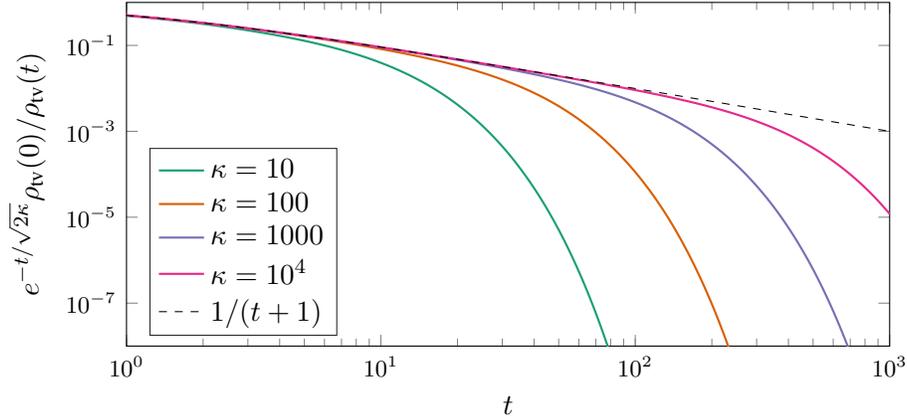}
\caption{The plot shows the convergence rate of the fundamental solutions \eqref{eq:fundSol} as a function of $\kappa$ for the choice $d_\infty=1/\sqrt{2\kappa}$, $t_0=0$, $d(0)=1$. For small values of $t$, the convergence rate is roughly $1/(1+t)$, whereas for larger values of $t$ the convergence rate is approximately $\exp(-t/\sqrt{2\kappa})$.}
\label{Fig:TVdynamics}
\end{figure}

Summarizing, we therefore conclude according to Proposition~\ref{Prop:mainTV} that choosing $d_\infty \sim 1/\sqrt{\kappa}$ implies that the trajectories of the time-varying generalization of \eqref{eq:exCT} satisfy the estimate
\begin{equation}
|(q(t),p(t))|\leq C_\text{tv} |(q(0),p(0))| \frac{\rho_\text{tv}(t)}{\rho_\text{tv}(0)} \begin{cases} e^{-d_\infty t}, & 0\leq d_\infty \leq 1/\sqrt{\kappa},\\
e^{(-d_\infty + \sqrt{d_\infty^2 -1/\kappa})t}, &d_\infty > 1/\sqrt{\kappa},\end{cases}
\end{equation}
for some constant $C_\text{tv}$ (that might strongly depend on $\kappa$) and for all $(q(0),p(0))\in A \subset \mathcal{A}_f$, where $A$ is compact. For any $\kappa\geq 1$, the right-hand side can be bounded by
\begin{equation}
|(q(t),p(t))|\leq \frac{C_\text{tv} |(q(0),p(0))| }{1+d(0) t},
\end{equation}
\hchange{which implies that the convergence rate is at least $\mathcal{O}(1/t)$ for any fixed function $f$ that has non-degenerate isolated local minima (the minima can have an arbitrarily small curvature). We believe that this provides useful insights for the case where the isolated minimum at the origin is degenerate. For example, given a function with an isolated minimum that is degenerate, we can add a regularization of the type $\epsilon |x|^2/2$, where $\epsilon>0$ is small. The bound (54) applies to the resulting regularized function, and implies a convergence rate of the order $\mathcal{O}(1/t)$.}

\hchange{It is important to notice that the results in this section are stated in terms of the distance between the current iterate and the optimizer, i.e., $|(q(t)-x^*,p(t))|$. In the case of smooth and convex functions, this yield the following bound on the function value:
\begin{equation*}
    f(q(t))-f(x^*) \leq L/2 |q(t)-x^*|^2 \leq \frac{C_{\text{tv}}^2}{(1+t)^2} |q(0)-x^*|^2,
\end{equation*}
provided that the algorithm is initialized with $p(0)=0$. Thus, our analysis recovers the well-known optimal rate $\mathcal{O}(1/t^2)$ (in terms of function value) for smooth and convex functions.}

\subsection{Example 2}\label{Sec:OptiEx2TV}
We consider the time-varying generalization of \eqref{eq:exDT}, where the parameters $d>0$ and $\beta\geq 0$ are assumed to be time varying. We focus on the case where $f$ is convex and $\beta_k$ is chosen to be $\beta_k=T(1-2 d_k T)$, as this simplifies the stability analysis and enables explicit closed-form evaluation of certain expressions. The stability analysis carried out in Appendix~\ref{App:DisStabilityConvex} applies likewise to the time-varying case; it suffices to replace $d$ with $d_k$ and $\beta$ with $\beta_k$. This concludes that the origin is asymptotically stable as long as $0<d_k T < 1$, $0<T\leq 1/\sqrt{L}$.

After a change of coordinates that diagonalizes $H_\text{e}$, the linearized dynamics of a single coordinate are given by
\begin{equation}
\delta q(k+1)=\delta q(k)+ T \delta p(k+1), \qquad \delta p(k+1)=\delta p(k)-T(2d_k+\beta_k h) \delta p(k) - T h \delta q(k), \label{eq:linDT}
\end{equation}
where $\delta q(k)\in \mathbb{R}, \delta p(k)\in \mathbb{R}$, and $h$ denotes the corresponding eigenvalue of $H_\text{e}$ with upper and lower bounds $1/\kappa \leq h \leq 1$. In analogy to the transformation \eqref{eq:cochange} in the continuous-time case, we define 
\begin{equation}
\left(-\frac{1}{2}\right)^{k_0+1} (1-hT^2)^{k-k_0-2} \left( \prod_{j=k_0+1}^{k-1} (1- d_{j} T) \right) \delta \tilde{q}(k):=\delta q(k), \label{eq:cochangeDT}
\end{equation}
which transforms \eqref{eq:linDT} into
\begin{equation}
\delta \tilde{q}(k+2) - 2 \delta \tilde{q}(k+1) + \frac{4 \beta_{k+1}/T }{(1-hT^2)(1+\beta_{k+1}/T)(1+\beta_{k}/T)} \delta \tilde{q}(k)=0. \label{eq:time-varyingHODT}
\end{equation}
This transformation is well known in the literature on linear difference equations~\citep[see, for example,][p.~369]{IntroDifferenceEquations}.
As in the continuous-time case, the analysis is considerably simplified if $\beta_k$ is chosen in such a way that the coefficient multiplying $\delta \tilde{q}(k)$ in \eqref{eq:time-varyingHODT} remains constant. This can be achieved, for example, with
\begin{equation}
4 \beta_{k+1}/T=\alpha_\text{tv} (1+\beta_{k+1}/T)(1+\beta_k/T), \label{eq:ricattiDT2}
\end{equation}
where $\alpha_\text{tv}$ is constant and independent of $h$. \hchange{This equation defines a recurrence relation for $\beta_k$, and through $\beta_k=T(1-2 d_k T)$ a relation for $d_k$. The choice \eqref{eq:ricattiDT2} is motivated by the fact that 1) the recurrence relation is independent of $h$ and 2) the difference equations \eqref{eq:time-varyingHODT} is time-invariant, which enables closed-form solutions.} 
Thus, any trajectory $\delta q(k)$ satisfying \eqref{eq:linDT} is a linear combination of the two fundamental solutions
\begin{equation}
(1-hT^2 \pm \sqrt{(1-h T^2)^2 - \alpha (1-hT^2)})^{k-k_0} \prod_{j=k_0+1}^{k-1} (1-d_j T). \label{eq:fundSolDT}
\end{equation}
This explicitly characterizes the convergence rate of \eqref{eq:linDT} and, by virtue of Proposition~\ref{Prop:mainTV}, the convergence rate of \eqref{eq:exDT}, (for time-varying $d_k$ and $\beta_k$). Similar to the continuous-time discussion in Section~\ref{Sec:OptiEx1TV}, the convergence rate can be improved by choosing $d_k$ to be close to one for small $k$. For the specific choice $T=0.5$, $\beta_0=0$, $d_0=1$, $\lim_{k\rightarrow \infty} d_k = d_\infty=1/\sqrt{2\kappa}$, $L=1$, the dependence of the convergence rate on $\kappa$ is illustrated in Figure~\ref{Fig:ConvRateDT}. As in the continuous-time case, the introduction of the time-varying parameter $d_k$ ensures that the fundamental solutions converge approximately with $1/(1+d_0 T k)$ for small $k$, whereas for large $k$ the convergence is linear with rate $1-T/\sqrt{2\kappa}$. The transition occurs for $k \sim \sqrt{2\kappa}/T$. This indicates that even for very large $\kappa$, the fundamental solutions converge roughly with $1/(1+d_0 T k)$.

\begin{figure}
\setlength{\figurewidth}{.7\columnwidth}
\setlength{\figureheight}{.3\columnwidth}
\center
\input{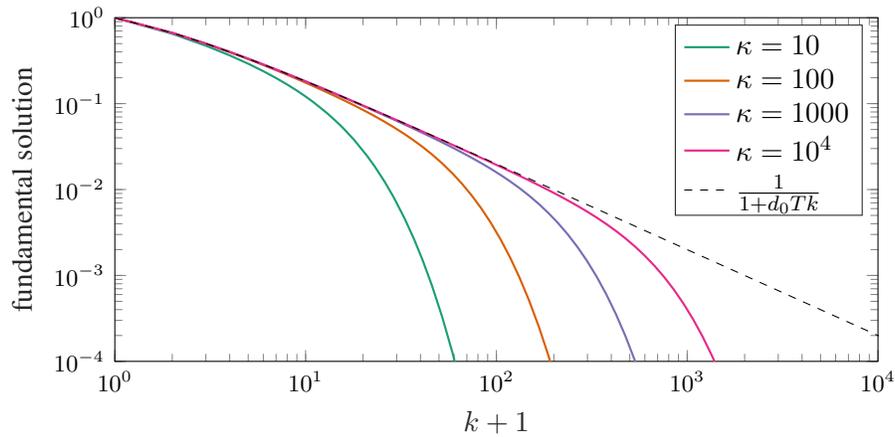}
\caption{The plot shows the convergence rate of the fundamental solutions \eqref{eq:fundSolDT} as a function of $\kappa$. For small values of $k$, the convergence rate is roughly $1/(1+d_0 T k)$, whereas for larger values of $k$ the convergence is linear with rate $1-T/\sqrt{2 \kappa}$.}
\label{Fig:ConvRateDT}
\end{figure}


\section{Conclusions}\label{Sec:Conclusion}
We have presented a convergence-rate analysis of momentum-based optimization algorithms from a dynamical systems point of view. Our analysis emphasizes the importance of the curvature properties about an isolated local minimum, which, up to a multiplicative constant, dictate the convergence rate. In addition, we find that reducing the damping over time improves the convergence rate by sublinear terms, which is important for objective functions that have minima with almost vanishing curvature. The use of momentum ensures robustness of the convergence rate against small changes in the curvature and leads, in many cases, to acceleration.

We also provided a rigorous motivation for  the use of symplectic discretization schemes, showing that they enable the computation of a modified energy function that is (almost exactly) preserved by the conservative parts of the dynamics. The modified energy function provides a means for stability analysis, which implies, for example, that certain heavy-ball-type methods are in fact accelerated. Thus, an evaluation of the gradient at a shifted position is not necessary for achieving convergence rates that scale with $1/\sqrt{\kappa}$ for large $\kappa$.

Our analysis emphasizes fundamental similarities between convergence rate analysis in continuous time and discrete time, and provides intuitive and rigorous explanations for various phenomena encountered in optimization, without resorting to convexity.



\acks{We thank the Branco Weiss Fellowship, administered by ETH Zurich, for the generous support and the Office of Naval Research under grant number N00014-18-1-2764. This work was also funded by the Deutsche Forschungsgemeinschaft (DFG, German Research Foundation)--MU 4710/2-1.}

\newpage


\appendix

\section{Proof of Proposition~\ref{Prop:main}}\label{App:PropMain}
For proving the proposition we rely on the following lemmas.
\begin{lemma}\label{Lem:Gronwall}
Consider the trajectories of a time-varying dynamical system
\begin{equation}
x^+(t)=(A+ B(t)) x(t),\quad \forall t\in I,
\end{equation}
with $A\in \mathbb{R}^{2n\times 2n}$ and $B: I \rightarrow \mathbb{R}^{2n\times 2n}$, where $B$ is either continuous (in the continuous-time case) or $A+B(t)$ is invertible (in the discrete-time case). Let $A$ be such that the trajectories $w^+(t)=A w(t)$ converge exponentially with rate $\alpha$, i.e., $|w(t)|\leq C_\text{w} |w(0)| \exp(-\alpha t)$, for all $w(0)\in \mathbb{R}^{2n}$, $t\in I$, and for some constant $C_\text{w}\geq 1$. Then, $x(t)$ satisfies the estimate
\begin{align}
|x(t)|&\leq C_\text{w} |x(0)| \exp\left(-\alpha t+ C_\text{w} \int_{0}^{t} |B(\tau)| \diff \tau\right), \qquad \text{resp.}\\
|x(t)|&\leq C_\text{w} |x(0)| \exp\left(-\alpha t+ C_\text{w} \exp(\alpha) \sum_{\tau=0}^{t-1} |B(\tau)|\right).
\end{align}
\end{lemma}
\begin{proof}
We consider the continuous-time case first. Due to the continuity assumptions on $B(t)$, the dynamics are guaranteed to have a unique solution, $x(t)$, satisfying
\begin{equation}
x(t) = \exp(A t) x(0) + \int_{0}^{t} \exp(A(t-\tau)) B(\tau) x(\tau) \diff \tau,
\end{equation}
where $\exp$ denotes the matrix exponential. The two-norm of $\exp(A (t-\tau))$ is, by assumption, upper bounded by $C_\text{w} \exp(-\alpha (t-\tau))$, for all $t,\tau \in I$, $t \geq \tau$, yielding the following estimate:
\begin{align}
\exp(\alpha t) |x(t)| \leq C_\text{w} |x(0)| + C_\text{w} \int_{0}^{t} |B(\tau)|  \exp(\alpha \tau) |x(\tau)| \diff \tau.\label{eq:tmp11}
\end{align}
Applying the Gr\"onwall inequality to \eqref{eq:tmp11} yields the desired result. 

The discrete-time case is analogous. Due to the assumption of $A+B(t)$ being full rank, the dynamics are guaranteed to have a unique solution $x(t)$, which satisfies
\begin{equation}
x(t)=A^t x(0)+\sum_{\tau=1}^{t} A^{t-\tau} B(\tau-1) x(\tau-1);
\end{equation}
see, for example, \citet[p.~59]{DifferenceEquations}. By assumption, the two-norm of $A^{t-\tau}$ is bounded by $C_\text{w} \exp(-\alpha(t-\tau))$, for all $t,\tau \in I$, $t\geq \tau$. As a result, the following estimate is obtained:
\begin{align}
\exp(\alpha t) |x(t)|&\leq C_\text{w} |x(0)| + C_\text{w} \sum_{\tau=1}^{t} \exp(\alpha \tau) |B(\tau-1)| |x(\tau-1)|\\
&=C_\text{w} |x(0)| + C_\text{w} \exp(\alpha) \sum_{\tau=0}^{t-1}  |B(\tau)| \exp(\alpha \tau) |x(\tau)|.
\end{align}
Applying the Gr\"onwall inequality \citep[p.~186]{DifferenceEquations}, yields
\begin{align}
|x(t)| &\leq C_\text{w} |x(0)| \exp(-\alpha t)~ \prod_{\tau=0}^{t-1} (1 + C_\text{w} \exp(\alpha) |B(\tau)|)\\
&\leq C_\text{w} |x(0)| \exp\left(-\alpha t + C_\text{w} \exp(\alpha) \sum_{\tau=0}^{t-1} |B(\tau)| \right).
\end{align}
\end{proof}

\begin{lemma}\label{Lem:AsymptConv}
Let Assumption~\ref{Ass:lin} be fulfilled. Then there exists an open ball $B_\delta$ (with radius $\delta>0$) centered at the origin and a constant $\tilde{C}\geq 1$, such that for all $z_0 \in B_\delta$,
\begin{equation}
|\varphi_t(z_0)| \leq \tilde{C} |z_0| \exp(-\alpha t), \quad \forall t\in I.
\end{equation}
\end{lemma}
\begin{proof}
We consider the continuous-time case first. Assumption~\ref{Ass:lin} implies that the dynamics \eqref{eq:ODE11} and \eqref{eq:ODE12} are (asymptotically) stable. This implies that for any $\varepsilon>0$ there exists a constant $\delta(\varepsilon) >0$ such that $|\varphi_t(z_0)|<\varepsilon$ for all $z_0\in \mathbb{R}^{2n}$ with $|z_0|<\delta(\varepsilon)$. For any $\varepsilon>0$, we consider the trajectory starting at $z_0\in \mathbb{R}^{2n}$, with $|z_0|<\delta(\varepsilon)$ and reformulate the dynamics by applying the mean value theorem:
\begin{align}
z^+(t)=\left.\frac{\partial g}{\partial z}\right|_{z=\xi(t)} z(t)=\left(\left.\frac{\partial g}{\partial z}\right|_{z=0} + \left.\frac{\partial g}{\partial z}\right|_{z=\xi(t)} - \left. \frac{\partial g}{\partial z}\right|_{z=0} \right) z(t),
\end{align}
where $\xi(t)\in \mathbb{R}^{2n}$ lies between the origin and $z(t)$. The continuity assumption on the dynamics leads to the following estimate:
\begin{equation}
\Big\vert\left.\frac{\partial g}{\partial z}\right|_{z=\xi(t)} - \left. \frac{\partial g}{\partial z}\right|_{z=0}\Big\vert \leq C_\text{g} |\xi(t)| \leq C_\text{g} |z(t)| < C_\text{g} \varepsilon, \label{eq:proofest11}
\end{equation}
for all $t\in I$ and where $C_\text{g}$ is a Lipschitz constant of $\partial g/\partial z$ in a neighborhood of the origin.

Applying Lemma~\ref{Lem:Gronwall} then yields
\begin{equation}
|z(t)|\leq C_\text{l} |z_0| \exp(-\alpha t + C_\text{l} C_\text{g} \varepsilon t ), \quad \forall t\in I.
\end{equation}
We fix $\varepsilon>0$ such that $C_\text{l} C_\text{g} \varepsilon < \alpha/2$.
This leads, in turn, to a refinement of the estimate \eqref{eq:proofest11}, 
\begin{equation}
\Big\vert\left.\frac{\partial g}{\partial z}\right|_{z=\xi(t)} - \left. \frac{\partial g}{\partial z}\right|_{z=0}\Big\vert \leq C_\text{l} C_\text{g} \delta(\varepsilon) \exp(-\alpha t/2)< \frac{\alpha}{2} \exp(-\alpha t/2), \quad \forall t\in I,
\end{equation}
where we have used the fact that $\delta(\varepsilon)\leq \varepsilon$. This results in
\begin{equation}
\int_{0}^{t} \Big\vert\left.\frac{\partial g}{\partial z}\right|_{z=\xi(\tau)} - \left. \frac{\partial g}{\partial z}\right|_{z=0}\Big\vert  \diff \tau < \frac{\alpha}{2} \int_{0}^{\infty} \exp(-\alpha \tau/2) \diff \tau = 1, \quad \forall t\in I. 
\end{equation}
Applying Lemma~\ref{Lem:Gronwall} once more therefore implies
\begin{equation}
|z(t)|\leq \underbrace{C_\text{l} e^{C_\text{l}}}_{:=\tilde{C}} |z_0| \exp(-\alpha t), \quad \forall t\in I,
\end{equation}
for all $z_0\in \mathbb{R}^{2n}$ with $|z_0|<\delta(\varepsilon)$, which leads to the desired result.

The same arguments (with slightly modified constants) apply to the discrete-time case.
\end{proof}


\section{Proof of Proposition~\ref{Prop:mainTV}}\label{App:PropMainTV}
For proving the proposition we rely on the following lemmas.
\begin{lemma}\label{Lem:GronwallTV}
Consider the trajectories of a time-varying dynamical system
\begin{equation}
x^+(t)=(A(t)+ B(t)) x(t),\quad \forall t\in I,
\end{equation}
with $A: I \rightarrow \mathbb{R}^{2n\times 2n}$ and $B: I \rightarrow \mathbb{R}^{2n\times 2n}$, where $A$ and $B$ are either continuous (in the continuous-time case) or $A(t)+B(t)$ is invertible for all $t\in I$ (in the discrete-time case). Let $\phi(t,\tau)\in \mathbb{R}^{2n\times 2n}$, $t,\tau\in I$, $t\geq \tau$, be defined by $\phi(t,\tau)^+= A(t) \phi(t,\tau)$ (for a fixed $\tau$) and $\phi(\tau,\tau)=I$, and assume $\phi(t,\tau)$ satisfies
\begin{equation*}
|\phi(t,\tau)|\leq C_\phi \frac{\rho(t)}{\rho(\tau)},
\end{equation*}
where $C_\phi \geq 1$ is constant, $\rho:I \rightarrow \mathbb{R}_{\geq 0}$ is continuous, monotonically decreasing, and $\lim_{t\rightarrow \infty} \rho(t)=0$. Then, $x(t)$ satisfies the estimate
\begin{align}
|x(t)|&\leq C_\phi |x(\tau)| \frac{\rho(t)}{\rho(\tau)} \exp\left(C_\phi \int_{\tau}^{t} |B(s)| \diff s\right), \qquad \text{resp.}\\
|x(t)|&\leq C_\phi |x(\tau)| \frac{\rho(t)}{\rho(\tau)}  \exp\left(C_\phi \sum_{j=\tau}^{t-1} \frac{\rho(j)}{\rho(j+1)} |B(j)|\right), \quad \forall t,\tau \in I, t\geq \tau.
\end{align}
\end{lemma}
\begin{proof}
We consider the continuous-time case first. Due to the continuity assumptions on $A(t)$ and $B(t)$, the dynamics are guaranteed to have a unique solution $x(t)$, satisfying
\begin{equation}
x(t) = \phi(t,\tau) x(\tau) + \int_{\tau}^{t} \phi(t,s) B(s) x(s) \diff s.
\end{equation}
The bound on the fundamental solution matrix $\phi(t,\tau)$ yields the following estimate:
\begin{align}
\frac{1}{\rho(t)} |x(t)| \leq C_\phi \frac{|x(\tau)|}{\rho(\tau)} + C_\phi \int_{\tau}^{t} |B(s)|  \frac{|x(s)|}{\rho(s)} \diff s.\label{eq:tmp11TV}
\end{align}
Applying the Gr\"onwall inequality to \eqref{eq:tmp11TV} yields the desired result. 

The discrete-time case is analogous. Due to the assumption of $A(t)+B(t)$ being full rank, the dynamics are guaranteed to have a unique solution $x(t)$, which satisfies
\begin{equation}
x(t)=\phi(t,\tau) x(\tau)+\sum_{j=\tau+1}^{t} \phi(t,j) B(j-1) x(j-1);
\end{equation}
see, for example, \citep[p.~59]{DifferenceEquations}. As a result, the bound on $\phi(t,\tau)$ leads to the following estimate
\begin{align}
\frac{1}{\rho(t)} |x(t)|&\leq C_\phi \frac{|x(\tau)|}{\rho(\tau)} + C_\phi \sum_{j=\tau+1}^{t} \frac{\rho(j-1)}{\rho(j)} |B(j-1)| \frac{|x(j-1)|}{\rho(j-1)}.
\end{align}
Applying the Gr\"onwall inequality \citep[p.~186]{DifferenceEquations}, yields
\begin{align}
|x(t)| &\leq C_\phi |x(\tau)| \frac{\rho(t)}{\rho(\tau)} ~ \prod_{j=\tau}^{t-1} (1 + C_\phi \frac{\rho(j)}{\rho(j+1)} |B(j)|)\\
&\leq C_\phi |x(\tau)| \frac{\rho(t)}{\rho(\tau)} \exp\left(C_\phi \sum_{j=\tau}^{t-1} \frac{\rho(j)}{\rho(j+1)} |B(j)| \right),
\end{align}
where the fact that $1+x\leq e^x$, for all $x\geq 0$ has been used.
\end{proof}

\begin{lemma}\label{Lem:AsymptConvTV}
Let Assumption~\ref{Ass:linTV} be fulfilled. Then there exists an open ball $B_\delta$ (with radius $\delta>0$) centered at the origin and a constant $\tilde{C}\geq 1$, such that for all $z_0 \in B_\delta$, and all $\tau\in I$,
\begin{equation}
|\varphi_t(z_0,\tau)| \leq \tilde{C} |z_0| \frac{\rho(t)}{\rho(\tau)} e^{-\alpha (t-\tau)}, \quad \forall t\in I, t\geq \tau.
\end{equation}
\end{lemma}
\begin{proof}
We consider the continuous-time case first and fix the initial time $\tau\in I$. Assumption~\ref{Ass:linTV} implies that the dynamics \eqref{eq:mombased2TV} are (asymptotically) stable, uniformly in $\tau$. This implies that for any $\varepsilon>0$ there exists a constant $\delta(\varepsilon) >0$ (independent of $\tau$) such that $|\varphi_t(z_0,\tau)|<\varepsilon$ for all $t\geq \tau$ and for all $z_0\in \mathbb{R}^{2n}$ with $|z_0|<\delta(\varepsilon)$. For any $\varepsilon>0$, we consider the trajectory starting at $z_0\in \mathbb{R}^{2n}$ at time $\tau \in I$, with $|z_0|<\delta(\varepsilon)$ and reformulate the dynamics by applying the mean value theorem:
\begin{align}
z^+(t)=\left.\frac{\partial g}{\partial z}\right|_{t,z=\xi(t)} z(t)=\left(\left.\frac{\partial g}{\partial z}\right|_{t,z=0} + \left.\frac{\partial g}{\partial z}\right|_{t,z=\xi(t)} - \left. \frac{\partial g}{\partial z}\right|_{t,z=0} \right) z(t),
\end{align}
where $\xi(t)\in \mathbb{R}^{2n}$ lies between the origin and $z(t)$. The assumptions on the dynamics, leads to the following estimate:
\begin{equation}
\Big\vert\left.\frac{\partial g}{\partial z}\right|_{t,z=\xi(t)} - \left. \frac{\partial g}{\partial z}\right|_{t,z=0}\Big\vert \leq C_\text{g}|\xi(t)| \leq C_\text{g} |z(t)| < C_\text{g} \varepsilon, \label{eq:proofest11tv}
\end{equation}
for all $t\in I$, $t\geq \tau$, and where $C_\text{g}$ is a time-independent Lipschitz constant of $\partial g/\partial z$ in a neighborhood of the origin.

By assumption, the linear dynamics decay at least with $\exp(-\alpha (t-\tau))$. Thus, applying Lemma~\ref{Lem:GronwallTV} yields
\begin{equation}
|z(t)|\leq C_\text{l} |z(\tau)| \exp(-\alpha (t-\tau) + C_\text{l} C_\text{g} \varepsilon (t-\tau) ), \quad \forall t\in I, \quad t\geq \tau.
\end{equation}
We fix $\varepsilon>0$ such that $C_\text{l} C_\text{g} \varepsilon < \alpha/2$.
This leads, in turn, to a refinement of the estimate \eqref{eq:proofest11tv}, 
\begin{equation}
\Big\vert\left.\frac{\partial g}{\partial z}\right|_{t,z=\xi(t)} - \left. \frac{\partial g}{\partial z}\right|_{t,z=0}\Big\vert \leq C_\text{l} C_\text{g}  \delta(\varepsilon) \exp(-\alpha (t-\tau)/2)< \frac{\alpha}{2} \exp(-\alpha (t-\tau)/2),
\end{equation}
$\forall t\in I$, $t\geq \tau$, where we have used the fact that $\delta(\varepsilon)\leq \varepsilon$. This results in
\begin{equation}
\int_{\tau}^{t} \Big\vert\left.\frac{\partial g}{\partial z}\right|_{t=\hat{t},z=\xi(\hat{t})} - \left. \frac{\partial g}{\partial z}\right|_{t=\hat{t},z=0}\Big\vert  \diff \hat{t} < \frac{\alpha}{2} \int_{\tau}^{\infty} \exp(-\alpha (\hat{t}-\tau)/2) \diff \hat{t} = 1, \quad \forall t\in I. 
\end{equation}
Applying Lemma~\ref{Lem:GronwallTV} once more therefore implies
\begin{equation}
|z(t)|\leq \underbrace{C_\text{l} e^{C_\text{l}}}_{:=\tilde{C}} |z(\tau)| \frac{\rho(t)}{\rho(\tau)} e^{-\alpha (t-\tau)}, \quad \forall t\in I, \quad t\geq \tau,
\end{equation}
for all $z_0\in \mathbb{R}^{2n}$ with $|z_0|<\delta(\varepsilon)$, which leads to the desired result.

The same arguments (with slightly modified constants) apply to the discrete-time case.
\end{proof}


\section{Proof of Proposition~\ref{Prop:ModEnergyShort}}\label{App:ProofTildeH}
We prove the following generalization of Proposition~\ref{Prop:ModEnergyShort}:
\begin{proposition}\label{Prop:ModEnergy}
Let $f$ be analytic on $B_r^\text{c}$, the closed ball of radius $r$ about the origin, and let $L_\text{H}$ be a Lipschitz constant of $\nabla H$, i.e., $|\nabla H(z)|\leq L_\text{H} |z|$ for all $z\in B_r^\text{c}\times B_r^\text{c}$. Then there exists a perturbed Hamiltonian $\tilde{H}: B_r^\text{c}\times B_r^\text{c} \rightarrow \mathbb{R}$, whose trajectories 
\begin{align}
&\dot{\tilde{q}}(t)=\frac{\partial \tilde{H}}{\partial p}\T, \quad
\dot{\tilde{p}}(t)=-\frac{\partial \tilde{H}}{\partial q}\T, \qquad (\tilde{q}(0),\tilde{p}(0))=z_0,	\label{eq:modODE}
\end{align}
are such that
\begin{equation}
|(\tilde{q}(T),\tilde{p}(T))-\Phi_T(z_0)|\leq T C_\text{E} |z_0| e^{-T_0/T},
\end{equation}
for all $0<T\leq T_0/3$ and all $z_0$ such that 
\begin{equation}
|z_0|\leq \frac{r}{2} e^{-10.2 L_\text{H}T}, \label{eq:condZ0}
\end{equation}
where $T_0$ and  $C_\text{E}$ are constants given by
\begin{equation*}
T_0:=\frac{2\text{ln}(2)-1}{2eL_\text{H}}, \quad C_\text{E}:=L_\text{H} (4.36e+e^2 T_0/3).
\end{equation*}
The perturbed Hamiltonian has the form
\begin{equation}
\tilde{H}(q,p)=H(q,p)-\frac{T}{2} \nabla f(q)\T p + T^2 F(q,p), \label{eq:HamForm2}
\end{equation}
where $\tilde{H}$ and $F:B_r^\text{c}\times B_r^\text{c} \rightarrow \mathbb{R}$ satisfy
\begin{align}
|\nabla H(z)-\nabla \tilde{H}(z)|&\leq 15 L_\text{H} T |\nabla H(z)|, \quad
|\nabla F(z)|\leq C_\text{F} |\nabla H(z)|, \quad \forall z\in B_{r/2}^\text{c} \times B_{r/2}^\text{c}, \label{eq:BoundTildeH}
\end{align}
with $C_\text{F}:=356 L_\text{H}^2$.
Moreover, $\tilde{H}$ has the same critical points as $H$ and satisfies
\begin{equation}
|H(z)-\tilde{H}(z)|\leq 15 L_\text{H} T |H(z)|, \quad |F(z)| \leq C_\text{F} |H(z)|, \quad \forall z\in \mathcal{A}_f\cap (B_{r/2}^\text{c} \times B_{r/2}^\text{c}). \label{eq:BoundTildeH2}
\end{equation}
Finally, at time $T$, the difference of the perturbed energy is bounded by
\begin{equation}
|\tilde{H}(\tilde{q}(T),\tilde{p}(T))-\tilde{H}(\Phi_T(z_0))|\leq T C_{\Delta \tilde{H}} |\nabla H(z_0)|^2 e^{-T_0/T}, \label{eq:boundH3}
\end{equation}
where
\begin{equation}
C_{\Delta \tilde{H}} := e(2.9+0.1 T_0) (1+e T_0/3), \quad |z_0|\leq \frac{r}{2} (1+3.63 L_\text{H} T (1+ e T_0/3))^{-1}. \nonumber
\end{equation}
\end{proposition}

\begin{proof}
We will complexify the position and momentum variables; i.e., we take $q\in \mathbb{C}^{n}$, $p\in \mathbb{C}^n$. Due to the fact that $f$ is analytic, it can be interpreted as a mapping from $\mathbb{C}^n$ to $\mathbb{C}$; a similar reasoning applies to $H$. Let $B_\rho^\text{c}\subset \mathbb{C}^{2n}$ be the closed ball of radius $\rho>0$, centered about the origin,  and let the norm $||\cdot ||_\rho$ be defined as 
\begin{equation}
||\bar{g}||_\rho:=\sup_{z\in B_\rho^\text{c}} |\bar{g}(z)|,
\end{equation}
for any analytic function $\bar{g}: \mathbb{C}^{2n} \rightarrow \mathbb{C}^{2n}$.

We define $f_1(q,p):=\Omega \partial H/\partial z=(p,-\nabla f(q))$, where $\Omega\in \mathbb{R}^{2n \times 2n}$ denotes the standard symplectic matrix, and introduce the following Ansatz for the differential equation \eqref{eq:modODE}:
\begin{equation}
\dot{\tilde{z}}(t)=f_1(\tilde{z}(t)) + f_2(\tilde{z}(t)) T + f_3(\tilde{z}(t)) T^2 + \dots,
\end{equation}
where $\tilde{z}(t):=(\tilde{q}(t),\tilde{p}(t))$ is a trajectory satisfying \eqref{eq:modODE}, and $f_i$ are analytic functions, $i=1,2,\dots$. Performing a Taylor expansion of $\tilde{z}(t)$ about $0$ and evaluating at $t=T$ yields, after rearranging terms, 
\begin{align}
\tilde{z}(T)&=\tilde{z}(0) + f_1(\tilde{z}(0)) T + (f_2(\tilde{z}(0)) + \frac{1}{2} D_1 f_1(\tilde{z}(0))) T^2 \\
&+ (f_3(\tilde{z}(0))+\frac{1}{3!} D_1^2 f_1(\tilde{z}(0)) + \frac{1}{2!} (D_2 f_1(\tilde{z}(0))+ D_1 f_2(\tilde{z}(0))) T^3 + \dots, \label{eq:tmpproof13}
\end{align}
where the Lie derivative $D_i: C^{\infty} \rightarrow C^{\infty}$ is defined by
\begin{equation}
D_i \bar{g}(z)= \left.\frac{\partial \bar{g}}{\partial z}\right|_{z} f_i(z), \quad \forall \bar{g}\in C^{\infty},
\end{equation}
where $C^{\infty}$ denotes the set of infinitely differentiable functions mapping from $\mathbb{C}^{2n}$ to $\mathbb{C}^{2n}$ \citep[see also][Ch.~IX]{Hairer}. The notation on the left-hand side of the previous equation should be read in the following order: The operator $D_i$ is applied to $\bar{g}$, yielding the function $D_i \bar{g}$, which is then evaluated at $z$. Applying the operator $D_i$ to the function $D_i \bar{g}$ is denoted by $D_i^2 \bar{g}$.

We require $\tilde{z}(T)=(q_{k+1},p_{k+1})$, which leads due to \eqref{eq:sympEuler} to
\begin{equation}
\tilde{z}(T)=\tilde{z}(0) + f_1(\tilde{z}(0)) T + \left( \begin{array}{c} -\nabla f(\tilde{q}(0))\\ 0 \end{array} \right) T^2.\label{eq:tmpproof12}
\end{equation}
Equating equal powers of $T$ in \eqref{eq:tmpproof13} and \eqref{eq:tmpproof12} yields the following recursive scheme for computing the functions $f_i$:
\begin{align}
f_2(z)&=(-\nabla f(q), 0 )\T-\frac{1}{2!} D_1 f_1 (z),\\
f_j(z)&= - \sum_{i=2}^{j} \frac{1}{i!} \sum_{k_1+\dots+k_i=j} D_{k_1} D_{k_2} \dots D_{k_{i-1}} f_{k_i}(z), \quad j>2,\label{eq:seqDef}
\end{align}
where the last sum ranges over all strictly positive integers $k_1,\dots, k_i$ that sum up to $j$. We will construct the perturbed Hamiltonian $\tilde{H}$ by appropriately truncating the series $\sum f_i(z)T^i$, and will show that the resulting truncated series has the desired properties.

We start by explicitly computing the function $f_2$:
\begin{equation}
f_2(z)=\left( \begin{array}{c} - \nabla f(q) \\ 0 \end{array} \right) - \frac{1}{2} \left( \begin{array}{cc} 0 & I\\ -\frac{\diff^2 f}{\diff x^2} &0 \end{array}\right) \left( \begin{array}{c} p \\ -\nabla f(q) \end{array} \right) = \underbrace{\frac{1}{2} \left( \begin{array}{cc} 0 & I \\ \frac{\diff^2 f}{\diff x^2} & 0\end{array}\right)}_{:=A_2(z)} f_1(z), \label{eq:defA2}
\end{equation}
and note that it is of the form $A_2(z) f_1(z)$, where $A_2(z)\in \mathbb{C}^{2n\times 2n}$. We further note that $f_2(z)$ is in fact a Hamiltonian vector field with corresponding Hamiltonian $-p\T \nabla f(q)/2$. From an induction argument relying on \eqref{eq:seqDef} and the definition of the Lie derivative, it follows that each $f_j$ can be written as $f_j(z)= A_j(z) f_1(z)$, where $A_1(z)$ is the identity, $A_2(z)$ is defined in \eqref{eq:defA2}, and
\begin{equation}
A_j(z):=-\sum_{i=2}^{j} \frac{1}{i!} \sum_{k_1+\dots + k_i=j} \frac{\partial}{\partial z}(D_{k_{2}} \dots D_{k_{i-1}} f_{k_i})\Big\vert_z A_{k_1}(z), \quad j>2.\label{eq:defA}
\end{equation}
Moreover, as shown in \citet[p.~295,Theorem~3.2]{Hairer}, the vector fields $f_j$ are guaranteed to be Hamiltonian for all $j\geq 1$. 

We will rely on the assumption that $f$ is analytic for bounding the functions $A_j$ using Cauchy's integral formula. More precisely, Cauchy's integral formula provides us with the following bound on $\partial \bar{g}/\partial z$ for any function $\bar{g}$ analytic in $B_r^\text{c}$:
\begin{equation}
\left\vert \left.\frac{\partial \bar{g}}{\partial z}\right|_{z} \right\vert \leq \frac{||\bar{g}||_{\rho}}{\delta}, \quad \forall z\in B_{\rho-\delta}^\text{c}, \qquad \left\vert\left\vert \frac{\partial \bar{g}}{\partial z} \right\vert\right\vert_{\rho-\delta} \leq \frac{||\bar{g}||_{\rho}}{\delta}, \label{eq:Cauchy}
\end{equation}
for any two constants $\rho,\delta$ such that $0\leq \delta < \rho\leq r$. A similar argument yields the following bound on the Lie-derivative of a function $\bar{g}$ analytic in $B_r^\text{c}$ \citep[p.~308]{Hairer},
\begin{equation}
\left\vert\left\vert D_i \bar{g} \right\vert\right\vert_{\sigma} \leq \frac{1}{\rho-\sigma} ||f_i||_{\sigma} ||\bar{g}||_{\rho},
\end{equation}
where $0\leq \sigma < \rho \leq r$. In the following these two inequalities are used to bound the right-hand side of \eqref{eq:defA}.

\emph{Claim:} The function $A_j$ is bounded above by 
\begin{equation}
||A_j||_{r/2} \leq \frac{\text{ln}(2)}{2\text{ln}(2)-1} \left( \frac{2 L_\text{H} (j-1)}{2 \text{ln}(2)-1}\right)^{j-1},
\end{equation}
for all $j\geq 1$.

\emph{Proof of the claim:}
Explicit calculations show that the claim holds for $j=1$. We thus fix the index $J>1$ and estimate $||A_j||_{r-(j-1)\delta}$ for all $1\leq j\leq J$, where $\delta=r/(2(J-1))$ is fixed, which, for $j=J$, yields the desired bound on $A_J$. To simplify notation we denote $||\cdot||_{r-(j-1)\delta}$ as $||\cdot||_{j}$. The right-hand side of \eqref{eq:defA} is upper bounded by
\begin{align*}
||\frac{\partial}{\partial z}(D_{k_{2}} \dots D_{k_{i-1}} f_{k_i})\Big\vert_z A_{k_1}||_j &\leq \frac{1}{\delta}||D_{k_2} \dots D_{k_{i-1}} f_{k_i}||_{j-1} ||A_{k_1}||_j\\
&\leq \frac{1}{\delta^2} ||D_{k_3} \dots D_{k_{i-1}} f_{k_i}||_{j-2} ||A_{k_2}||_{j-1} ||f_1||_{j-1} ||A_{k_1}||_j\\
&\leq \frac{1}{\delta^{i-1}} ||A_{k_i}||_{j-(i-1)} ||f_{1}||_{j-(i-1)} \dots ||A_{k_2}||_{j-1} ||f_1||_{j-1} ||A_{k_1}||_j.
\end{align*}
Due to the fact that $k_1+ \dots + k_i=j$ and $k_1>0, \dots k_i>0$, it follows $k_i \leq j-(i-1)$, which implies $||\bar{g}||_{j-(i-1)} \leq ||\bar{g}||_{k_i}$ for all $i\leq j$ and for any function $\bar{g}$ analytic on $B_r^\text{c}$. Therefore, the above bound reduces to
\begin{align}
||\frac{\partial}{\partial z}(D_{k_{2}} \dots D_{k_{i-1}} f_{k_i})\Big\vert_z A_{k_1}||_j &\leq \frac{1}{\delta^{i-1}} ||A_{k_i}||_{k_i} ||f_{1}||_{k_i} \dots ||A_{k_2}||_{k_2} ||f_1||_{k_2} ||A_{k_1}||_{k_1}\\
&\leq \left(\frac{||f_1||_r}{\delta}\right)^{i-1} ||A_{k_i}||_{k_i} \dots ||A_{k_1}||_{k_1}.
\end{align}
By introducing the constant $\hat{\delta}:=\delta/r$, which, by definition of $\delta$ satisfies $\hat{\delta}\leq 1/2$, we can reformulate the above bound as
\begin{equation}
||\frac{\partial}{\partial z}(D_{k_{2}} \dots D_{k_{i-1}} f_{k_i})\Big\vert_z A_{k_1}||_j\leq  \left(\frac{L_\text{H}}{\hat{\delta}}\right)^{i-1} ||A_{k_i}||_{k_i} \dots ||A_{k_1}||_{k_1}. \label{eq:tmp122}
\end{equation}
By exploiting the structure of $A_2(z)$ it can be verified that 
\begin{equation}
||A_2||_2 = \frac{1}{2} \max\{1,||\diff^2 f/\diff x^2||_2\}= \frac{1}{2} L_\text{H}\leq \frac{L_\text{H}}{2\hat{\delta}},\label{eq:tmp123}
\end{equation}
which will simplify the following development. Following \citet[p.~309]{Hairer}, we introduce the constants $b_j$ for all $j\geq 1$ in the following way
\begin{align*}
b_1=\frac{L_\text{H}}{\hat{\delta}}, \quad b_2=\frac{L_\text{H}^2}{\hat{\delta}^2},\quad 
b_j=\sum_{i=2}^{j} \frac{1}{i!} \sum_{k_1+\dots + k_i=j} b_{k_1} b_{k_2} \dots b_{k_i},\quad j>2.
\end{align*}
It can be verified with an induction argument taking \eqref{eq:tmp122} and \eqref{eq:tmp123} into account, that 
\begin{equation}
||A_j||_j \frac{L_\text{H}}{\hat{\delta}} \leq b_j, \quad 1\leq j\leq J.
\end{equation}
The variables $b_j$ are well-defined for any $j\geq 1$ and not just $1\leq j\leq J$. In order to bound the constants $b_j$, we formally introduce the generating function $b(\zeta)$, $b: \mathbb{C} \rightarrow \mathbb{C}$, $\zeta \in \mathbb{C}$, and note that $b$ satisfies
\begin{align}
b(\zeta):=\sum_{j=1}^{\infty} b_j \zeta^j&=b_1 \zeta +  \sum_{j=2}^{\infty} \zeta^j \sum_{i=2}^{j} \frac{1}{i!} \sum_{k_1+\dots+k_i=j} b_{k_1} \dots b_{k_i}  \label{eq:prooftmptmp}\\
&=b_1 \zeta + \sum_{i=2}^{\infty} \frac{b(\zeta)^i}{i!} \\
&=b_1 \zeta + e^{b(\zeta)} - 1 - b(\zeta),
\end{align}
where the order of summation has been interchanged to arrive at \eqref{eq:prooftmptmp}.
It can be verified by means of the implicit function theorem that, given any $v\in \mathbb{C}$, the equation $2b-\exp(b)+1=v$ can be uniquely solved for $b$ as long as $\exp(b)\neq 2$. This implies that $b(\zeta)$ is well-defined and analytic for $|b_1 \zeta|\leq 2 \text{ln}(2)-1$, or equivalently, $|\zeta|\leq (2 \text{ln}(2)-1)/b_1$. Furthermore, explicit calculations show that $|b(\zeta)|$ is bounded by $\text{ln}(2)$ for $|\zeta|\leq (2 \text{ln}(2)-1)/b_1$ \cite[see][p.~310]{Hairer}. Cauchy's inequality therefore implies
\begin{equation}
|b_j|\leq \frac{\text{ln}(2)}{ ((2\text{ln}(2)-1)/b_1)^{j}}=\text{ln}(2) \left(\frac{L_\text{H}}{\hat{\delta}(2 \text{ln}(2)-1)}\right)^j, \quad \forall j\geq 1.
\end{equation}
Evaluating the above bound for $j=J$ yields 
\begin{equation}
||A_J||_J \leq \frac{\text{ln}(2)}{2 \text{ln}(2)-1} \left( \frac{2L_\text{H} (J-1)}{2 \text{ln}(2)-1}\right)^{J-1},
\end{equation}
and proves the claim.

We define 
\begin{equation}
\Omega \nabla \tilde{H}(z):=\sum_{j=1}^{N} f_j(z) T^{j-1}, \quad z\in B_{r/2}^\text{c},
\end{equation}
where $N$ is chosen to be the largest integer such that $N T \leq T_0$, with $T_0:=(2\text{ln}(2)-1)/(2eL_\text{H})$. The choice for the integer $N$ is motivated by the following observation: Due to the bound on $||A_j||_{r/2}$ the terms in the above sum behave as $(Tj/(eT_0))^j$, which attains a minimum for $j\approx T_0/T$. The fact that all $f_j$, $j\geq 1$, are Hamiltonian vector fields proves \eqref{eq:HamForm2}.

As a result, we obtain for $z\in B_{r/2}^\text{c}$ and $T\leq T_0$ (implying $T/T_0\leq 1/N$ by our choice of $N$)
\begin{align}
|\nabla \tilde{H}(z)-\nabla H(z)|&\leq \sum_{j=2}^{N} |f_j(z)| T^{j-1} \leq \frac{\text{ln}(2)}{2\text{ln}(2)-1} |\nabla H(z)| \sum_{j=2}^{N} \left( \frac{T (j-1)}{e T_0} \right)^{j-1} \nonumber\\
&\leq \frac{\text{ln}(2)}{2\text{ln}(2)-1} \frac{T}{T_0} |\nabla H(z)| \sum_{j=2}^{N} \left(\frac{j-1}{e}\right)^{j-1} \left(\frac{T}{T_0}\right)^{j-2} \nonumber\\
&\leq \frac{\text{ln}(2)}{2\text{ln}(2)-1} \frac{T}{T_0} |\nabla H(z)| \underbrace{\sum_{j=1}^{N-1} (j/e)^j (1/N)^{j-1}}_{\leq 0.588} \nonumber
\\
&\leq 15 L_\text{H} T |\nabla H(z)|,\label{eq:boundAbove}
\end{align}
which proves the first bound in \eqref{eq:BoundTildeH}. The second bound is proved analogously; that is,
\begin{align}
|\nabla F(z)|&\leq \frac{1}{T^2} \sum_{j=3}^{N} |f_j(z)| T^{j-1} \\
&\leq \frac{\text{ln}(2)}{(2\text{ln}(2)-1)T_0^2} |\nabla H(z)| \sum_{j=3}^{N} \left(\frac{j-1}{e}\right)^{j-1} \left(\frac{T}{T_0}\right)^{j-3}\\
&\leq \frac{\text{ln}(2)}{(2\text{ln}(2)-1)T_0^2} |\nabla H(z)| \underbrace{\sum_{j=2}^{N-1} \left(\frac{j}{e}\right)^{j} \left(\frac{1}{N}\right)^{j-2}}_{\leq 1}\\
&\leq 356 L_\text{H}^2 |\nabla H(z)|.
\end{align}
The fact that $\tilde{H}$ has the same critical points is derived from the following observation. First, all critical points of $H$ are necessarily critical points of $\tilde{H}$, which follows from \eqref{eq:boundAbove} with $\nabla H(z)=0$, and second, $\nabla \tilde{H}(\hat{z})=0$ implies, according to \eqref{eq:boundAbove}, that $|\nabla H(\hat{z})|\leq 15 L_\text{H} T |\nabla H(\hat{z})| \leq 0.36 |\nabla H(\hat{z})|$ (for $T\leq T_0/3$), concluding $\nabla H(\hat{z})=0$. 

The first bound in \eqref{eq:BoundTildeH2} is derived by noticing that
\begin{equation}
\tilde{H}(z)-H(z) = \int_{0}^{1} (\nabla \tilde{H}(\gamma(t)) - \nabla H(\gamma(t)))\T \dot{\gamma}(t) \diff t,
\end{equation}
for any path $\gamma: [0,1] \rightarrow \mathbb{C}^{2n}$ that connects the origin with $z$; i.e., $\gamma(0)=0$, $\gamma(1)=z$. We choose $\gamma(t)$ such that $\nabla H(\gamma(t))\T \dot{\gamma}(t) = |\nabla H(\gamma(t))| |\dot{\gamma}(t)|$. This can be done by considering the gradient flow $\tilde{\gamma}(t)$,
\begin{equation}
\dot{\tilde{\gamma}}(t) = -\nabla H(\tilde{\gamma}(t)), \quad \tilde{\gamma}(0)=z,
\end{equation}
which is guaranteed to converge to the origin for any $z\in \mathcal{A}_f \cap (B_{r/2}^\text{c}\times B_{r/2}^\text{c})$. As a result, choosing $\gamma(t)=\tilde{\gamma}(1/t-1)$ leads to 
\begin{align}
|\tilde{H}(z)-H(z)|&\leq \int_{0}^{1} |\nabla \tilde{H}(\gamma(t)) - \nabla H(\gamma(t))| |\dot{\gamma}(t)| \diff t \nonumber \\
&\leq 15 L_\text{H} T \int_{0}^{1} |\nabla H(\gamma(t))| |\dot{\gamma}(t)| \diff t \nonumber \\
&=15 L_\text{H} T \int_{0}^{1} \nabla H(\gamma(t))\T \dot{\gamma}(t) \diff t
= 15 L_\text{H}T H(z).
\end{align}
The same argument (with the same path) applies to $|\nabla F(z)|$, which yields the second bound in \eqref{eq:BoundTildeH2}.

We show next that the solution of \eqref{eq:modODE} is exponentially close to the iterate \eqref{eq:sympEuler}. We fix $T$, $z_0\in B_{r/2}^\text{c}$ such that
\begin{equation}
|z_0|\leq r e^{-10.2 L_\text{H} T}/2 \label{eq:tmpproof121}
\end{equation}
(the choice becomes evident), and as above, $N$ to be the largest integer such that $NT \leq T_0$. For small enough $\zeta \in \mathbb{C}$, we denote the map from $\tilde{z}(0)=z_0$ to $\tilde{z}(t=\zeta)$, where $\tilde{z}(t)$ satisfies $\dot{\tilde{z}}(t)=\sum_{j=1}^{N} f_j(\tilde{z}(t)) \zeta^j$, by $\tilde{\varphi}_{\zeta}: \mathbb{C}^{2n} \rightarrow \mathbb{C}^{2n}$. The symplectic Euler step $\Phi_\zeta(z_0)$ and $\tilde{\varphi}_\zeta(z_0)$ are both analytic in $\zeta$ and by the above construction, their Taylor expansion in $\zeta$ about $\zeta=0$ agrees for the first $N$ terms ($z_0$ is fixed). 
Thus, the function $\tilde{g}(\zeta)/\zeta^{N+1}$, where $\tilde{g}(\zeta):=\Phi_\zeta(z_0)-\tilde{\varphi}_\zeta(z_0)$ is analytic. Applying the maximum modulus principle implies
\begin{equation}
\frac{|\tilde{g}(T)|}{T^{N+1}} \leq \max_{|\zeta|\leq \epsilon} \frac{|\tilde{g}(\zeta)|}{|\epsilon|^{N+1}} = \max_{|\zeta|=\epsilon} \frac{|\tilde{g}(\zeta)|}{\epsilon^{N+1}} = \frac{1}{\epsilon^{N+1}} \max_{|\zeta|=\epsilon} |\tilde{g}(\zeta)|, \label{eq:maxPrinciple}
\end{equation}
where $\epsilon$ is chosen as $\epsilon=eT_0/N$ (again, the choice becomes evident subsequently). It remains to bound $|\tilde{g}(\zeta)|$ for $|\zeta|=\epsilon$, which will be done by estimating $|\Phi_\zeta(z_0)-z_0|$ and $|\tilde{\varphi}_\zeta(z_0)-z_0|$ separately. For the first term we obtain (by definition of \eqref{eq:tmpproof12})
\begin{equation}
|\Phi_\zeta(z_0)-z_0|\leq |\zeta| |f_1(z_0)| (1 + |\zeta|)\leq \epsilon |f_1(z_0)| (1+\epsilon)\leq \epsilon (1+e T_0/3)L_\text{H} |z_0|, \label{eq:tmpproof1211}
\end{equation}
where $N\geq 3$ has been used for the last inequality. For bounding the second term we first derive a bound on $\nabla \tilde{H}(z)$, for all $z\in B_{r/2}^\text{c}$, similar to \eqref{eq:boundAbove} (unlike \eqref{eq:boundAbove} we have $\epsilon=eT_0/N$). We obtain
\begin{align}
|\nabla \tilde{H}(z)| &\leq  |\nabla H(z)| \left(1+\sum_{j=2}^{N} ||A_j||_{r/2} \epsilon^{j-1}\right) \nonumber\\
&\leq |\nabla H(z)| \left(1+\frac{\text{ln}(2)}{2\text{ln}(2)-1} \sum_{j=2}^{N} \left( \frac{2L_\text{H}\epsilon(j-1)}{2\text{ln}(2)-1} \right)^{j-1}\right)\nonumber\\
&= |\nabla H(z)| \left(1+\frac{\text{ln}(2)}{2\text{ln}(2)-1} \sum_{j=1}^{N-1} (j/N)^j \right)\leq 2.8 |\nabla H(z)|,
\end{align}
for all $z\in B_{r/2}^\text{c}$. The gradient of $\nabla \tilde{H}$ has therefore a Lipschitz constant of less than 2.8$L_\text{H}$ for $z\in B_{r/2}^\text{c}$, which implies that 
\begin{equation}
|\tilde{\varphi}_{\epsilon}(z_0)|\leq |z_0| e^{2.8 L_\text{H} \epsilon} \leq |z_0| e^{10.2 L_\text{H}T},
\end{equation}
as long as the corresponding trajectory stays within $B_{r/2}^\text{c}$. Note that the last inequality stems from the fact that $T_0<T(N+1)$ and thus $\epsilon<e T (N+1)/N \leq 4eT/3$. Due to the condition \eqref{eq:tmpproof121}, the trajectory $\tilde{\varphi}_\epsilon(z_0)$ necessarily stays within $B_{r/2}^\text{c}$. Moreover, from $\epsilon \leq e T_0/3$ ($N\geq 3$ since $T\leq T_0/3$), it can be inferred that 
\begin{equation}
|\tilde{\varphi}_\epsilon(z_0)| \leq |z_0| e^{2.8 e T_0 L_\text{H}/3} \leq 1.2 |z_0|,
\end{equation}
which implies that $|\tilde{\varphi}_\epsilon(z_0) - z_0|\leq 3.36 \epsilon L_\text{H} |z_0|$.
Combined with \eqref{eq:tmpproof1211}, this leads to $|\tilde{g}(\zeta)|\leq \epsilon \tilde{C}_\text{g}$ for $|\zeta|=\epsilon$ and $\tilde{C}_\text{g}:=(4.36+eT_0/3)L_\text{H}|z_0|$, which, according to \eqref{eq:maxPrinciple}, implies
\begin{align}
|\tilde{g}(T)|&\leq \epsilon \tilde{C}_\text{g} \left(\frac{T}{\epsilon}\right)^{N+1}\leq T \tilde{C}_\text{g} \left(\frac{T}{\epsilon}\right)^{N} \leq T \tilde{C}_\text{g} \left( \frac{N T}{e T_0} \right)^{N} \leq  T \tilde{C}_\text{g} e^{-N} \nonumber\\
&=  T \tilde{C}_\text{g} e e^{-(N+1)} \leq T \tilde{C}_\text{g} e e^{-T_0/T}. \label{eq:tmpproof1222}
\end{align}
The fact that $N\leq T_0/T < N+1$ has been used to obtain the last inequality. The above result justifies the choice $\epsilon=eT_0/N$.

We derive the bound on $\tilde{H}(\tilde{\varphi}_T(z_0))-\tilde{H}(\Phi_T(z_0))$ in a similar way. Consider the analytic function $\hat{g}(\zeta)=\tilde{H}(\tilde{\varphi}_\zeta(z_0))-\tilde{H}(\Phi_\zeta(z_0))$, where $z_0\in B_{r/2}^\text{c}$, $T$, and $N$ are fixed. The modified Hamiltonian $\tilde{H}$ is dependent on $\zeta$ and has the form
\begin{equation}
\tilde{H}(z)=\sum_{j=1}^{N} H_j(z) \zeta^{j-1},
\end{equation} 
where $H_1=H$, and $\Omega \nabla H_j(z)=f_j(z)$ for all $j$ with $1\leq j \leq N$.
The function $\hat{g}(\zeta)/(\zeta^{N+1})$ is analytic, as can be seen, for example, by a Taylor expansion of $\tilde{H}$ in $z$ about $z=\Phi_\zeta(z_0)$, where $\tilde{\varphi}_\zeta(z_0)$ and $\Phi_\zeta(z_0)$ agree up to $N$th order. We invoke the maximum principle, similar to \eqref{eq:maxPrinciple}, which implies,
\begin{equation}
\frac{|\hat{g}(T)|}{T^{N+1}} \leq \frac{1}{\epsilon^{N+1}} \max_{|\zeta|=\epsilon} |\hat{g}(\zeta)|,
\end{equation}
and where $\epsilon$ is again set to $\epsilon=eT_0/N$. A bound for $\hat{g}(\zeta)$ is obtained in the following way (applying Taylor's theorem):
\begin{align}
|\tilde{H}(\tilde{\varphi}_\zeta(z_0))-\tilde{H}(\Phi_\zeta(z_0))|&=|\tilde{H}(z_0)-\tilde{H}(\Phi_\zeta(z_0))| \\
&\leq |\nabla \tilde{H}(z_0)| |\Phi_\zeta(z_0)-z_0| + \frac{1}{2} \left\vert \left.\frac{\partial^2 \tilde{H}}{\partial z^2}\right|_{\eta} \right\vert |\Phi_\zeta(z_0)-z_0|^2,\label{eq:prooftmptaylor}
\end{align}
where $\eta$ lies between $\Phi_{\zeta}(z_0)$ and $z_0$. The gradient of $\tilde{H}$ is necessarily bounded by $|\nabla \tilde{H}(z)| \leq 2.8 L_\text{H} |z|$, which therefore implies that the Hessian of $\tilde{H}$ in the expression is bounded by $2.8 L_\text{H}$ provided that $\eta$ is in $B_{r/2}^\text{c}$. The bound \eqref{eq:tmpproof1211} implies that this is the case for 
\begin{equation}
|z_0|\leq \frac{r}{2} (1+ 4 L_\text{H} eT (1+e T_0/3)/3)^{-1},
\end{equation}
where $\epsilon \leq eT_0/3$ and $\epsilon < 4eT/3$ have been used. 
Combining \eqref{eq:prooftmptaylor} with the bound from \eqref{eq:tmpproof1211} yields
\begin{align}
|\tilde{H}(\tilde{\varphi}_\zeta(z_0))-\tilde{H}(\Phi_\zeta(z_0))|&\leq 2.8 \epsilon (1+ eT_0/3) |\nabla H(z_0)|^2 (1+ L_\text{H} \epsilon/2 (1+eT_0/3))\\
&\leq \epsilon |\nabla H(z_0)|^2 (2.8 + 0.1 (1+eT_0/3))(1+eT_0/3)\\
&\leq \epsilon |\nabla H(z_0)|^2 (2.9 + 0.1 T_0)(1+eT_0/3).
\end{align}
Applying the same chain of arguments as in \eqref{eq:tmpproof1222} allows us to conclude:
\begin{equation}
|\hat{g}(T)|\leq (2.9+0.1 T_0)(1+eT_0/3) e T e^{-T_0/T} |\nabla H(z_0)|^2.
\end{equation}
\end{proof}



\section{Proof of Proposition~\ref{Prop:StabilityDT}}\label{App:ProofDTStab}
\begin{proof}
We choose $r>0$ such that $A \subset B_r^{\text{c}} \times B_r^{\text{c}}$, where $B_r^\text{c}$ is the closed $n$-dimensional ball of radius $r$ centered at the origin. We use the Stone-Weierstrass theorem \citep[p.~159]{RudinAnalysis} to approximate $\diff^2 f/\diff x^2$ on $B_{2r}^{\text{c}}$ by $\diff^2 \tilde{f}/\diff x^2$ such that $|\diff^2 \tilde{f}/\diff x^2 -\diff^2 f/\diff^2 x|\leq \epsilon/(2r)$ for all $x\in B_{2r}^{\text{c}}$ and 
\begin{equation*}
\left.\frac{\diff^2 \tilde{f}}{\diff x^2}\right|_{x=0}=\left.\frac{\diff^2 f}{\diff x^2}\right|_{x=0},
\end{equation*}
where $\epsilon>0$ will be chosen subsequently. We integrate $\diff^2 \tilde{f}/\diff x^2$ and impose $\nabla \tilde{f}(0)=0$, which yields the function $\nabla \tilde{f}$ on $B_{2r}^\text{c}$ that is $\epsilon$-close to $\nabla f$. We integrate once more and impose $\tilde{f}(0)=0$. Moreover, by choosing $\epsilon$ small enough, the origin is guaranteed to be the only critical point of $H(z)$ (with $f$ replaced with $\tilde{f}$) for all $z\in A$, since $A\subset \mathcal{A}_f$ and $A\subset B_r^\text{c}\times B_r^\text{c}$. Due to the fact that $\tilde{f}$ is a polynomial and therefore analytic, we can invoke Proposition~\ref{Prop:ModEnergyShort} with $f$ replaced by $\tilde{f}$. Next, we will show that the origin is an asymptotically stable equilibrium of the dynamics \eqref{eq:exDT}, where $f$ is replaced by $\tilde{f}$, with region of attraction at least $A$.

In order to simplify notation, we reformulate the contraction step $\Phi_{d,T}$ (with $f$ replaced by $\tilde{f}$) in the following way:
\begin{equation}
\Phi_{d,T}(q_k,p_k)=\left( \begin{array}{c} q_k \\ (I-T\Lambda(q_k,p_k)) p_k \end{array} \right),
\end{equation}
where by assumption $\Lambda(q_k,p_k)$ is symmetric and positive definite. Due to the assumption on the \hchange{dissipative forces}, the singular values of $\Lambda(q_k,p_k)$ are upper and lower bounded by $d_2$ and $d_1$, respectively. Provided that the time step $T$ is small enough, $T\leq T_\text{max}$, we conclude from Proposition~\ref{Prop:ModEnergyShort} that there exists a modified energy function $\tilde{H}$ that has the same critical points as $H$, and such that the bound \eqref{eq:boundH3Short} holds for all $z_0 \in B_{r}^\text{c}\times B_{r}^\text{c}$. Morse theory implies that $\tilde{H}(z)$ is locally positive definite and has compact level sets for $z\in A$. The same applies therefore to the function
\begin{equation}
V(q_k,p_k):=\tilde{H}(q_k,p_k) + \frac{T d_1}{2} \nabla \tilde{f} (q_k)\T p_k,
\end{equation}
provided that $T\leq T_{\text{max}}$, where $T_\text{max}$ is chosen to be small enough (independently of $f_\text{d}$; this can be done since the positive definiteness of $\tilde{H}$ is not affected by the \hchange{dissipative forces} $f_\text{d}$).

The stability analysis simplifies when analyzing $\Phi_{d,T} \circ \Phi_T$ instead of $\Phi_T \circ \Phi_{d,T}$, due to the simple structure of $\Phi_{d,T}$. In order to simplify the notation we hide all constants that may depend on $d_2/d_1$, an upper bound on $d_2$, $L_\text{H}$, or higher orders of $T$ with the notation $\mathcal{O}(\cdot)$. It will be shown that $V(q_k,p_k)$ necessarily decays along the step $\Phi_{d,T} \circ \Phi_T$. To that extent, we define $z_2=\Phi_{d,T}(z_1)$ and $z_1=\Phi_T(z_0)$ and consider $V(z_2)-V(z_0)$, which, due to the fact that $q_2=q_1$, is bounded by
\begin{align*}
V(q_2,p_2)&=\frac{1}{2} |p_2|^2 - \frac{T}{2} \nabla \tilde{f}(q_2)\T (I-d_1 I) p_2 + \tilde{f}(q_2) + T^2 F(q_2,p_2)\\
&\leq \frac{1}{2} |p_1|^2 - T p_1\T \Lambda p_1 - \frac{T}{2} \nabla \tilde{f}(q_1)\T (I- d_1 I) p_1+\tilde{f}(q_1) + T^2 F(q_1,p_2)\\
&\hspace{3cm}+\frac{T^2}{2} \nabla \tilde{f}(q_1)\T (I-d_1 I) \Lambda p_1 + \mathcal{O}(d_1^2 T^2) |p_1|^2,\\
&\leq \tilde{H}(q_1,p_1)- T p_1\T \Lambda p_1 + \frac{T d_1}{2} \nabla \tilde{f}(q_1)\T p_1 \\
&\hspace{3cm}+\mathcal{O}(d_1 T^2) |p_1|^2 + \mathcal{O}(d_1 T^2)|\nabla \tilde{f}(q_1)| |p_1|,
\end{align*}
where the arguments of $\Lambda$ have been dropped and the results from Proposition~\ref{Prop:ModEnergyShort} have been used, which enable the following bound:
\begin{align*}
F(q_2,p_2)-F(q_1,p_1)&\leq \nabla F(q_1,p_1)\T (p_2-p_1) + \frac{1}{2} |p_2-p_1|^2 \sup_{z\in B_r} |\frac{\diff^2 F}{\diff z^2}|\\
&\leq 356 L_\text{H}^2 T |\nabla H(q_1,p_1)| |\Lambda p_1| + 178 L_\text{H}^3 T^2 |\Lambda p_1|^2 \\
&\leq \mathcal{O}(d_1 T) (|p_1|^2 +  |p_1| |\nabla \tilde{f}(q_1)|) + \mathcal{O}(T^2 d_1^2) |p_1|^2.
\end{align*}
Due to Proposition~\ref{Prop:ModEnergyShort}, it follows that $\tilde{H}(\Phi_{T}(z_0))-\tilde{H}(z_0)$ is bounded by a term of the order $|\nabla H(z_0)|^2 T e^{-T_0/T}$. This yields
\begin{align*}
V(q_2,p_2) \leq&\ \tilde{H}(q_0,p_0) - T p_1\T \Lambda p_1 + \frac{Td_1}{2} \nabla \tilde{f}(q_1)\T p_1 + |\nabla H(z_0)|^2 \mathcal{O}(T e^{-T_0/T})\\
& + \mathcal{O}(d_1 T^2) |p_1|^2 + \mathcal{O}(d_1 T^2)|\nabla \tilde{f}(q_1)| |p_1|\\
\leq&\ V(q_0,p_0) - \frac{Td_1}{2} \nabla \tilde{f}(q_0)\T p_0 - T p_1\T \Lambda p_1 +\frac{Td_1}{2}\nabla \tilde{f}(q_1)\T p_1 + |\nabla H(z_0)|^2 \mathcal{O}(T e^{-T_0/T})\\
&+ \mathcal{O}(d_1 T^2) |p_1|^2 +\mathcal{O}(d_1 T^2) |\nabla \tilde{f}(q_1)| |p_1|.
\end{align*}
Due to the continuity of $\nabla \tilde{f}$ we can bound $|\nabla \tilde{f}(q_1)-\nabla \tilde{f}(q_0)|$ by $T (L+\epsilon) |p_1|$, where $L$ is the Lipschitz constant of $\nabla f$. Moreover, $L_\text{H}$ is an upper bound on $L$. This implies
\begin{align*}
V(q_2,p_2)&\leq V(q_0,p_0) - T p_1\T \Lambda p_1 -\frac{T^2 d_1}{2} |\nabla \tilde{f}(q_0)|^2 + |\nabla H(z_0)|^2 \mathcal{O}(T e^{-T_0/T}) + \mathcal{O}(d_1 T^2) |p_1|^2\\
&\hspace{3cm}+\mathcal{O}(d_1 T^2) |\nabla \tilde{f}(q_1)||p_1|,\\
&\leq V(q_0,p_0) - \frac{T}{2} p_1\T \Lambda p_1 -\frac{T^2 d_1}{4} |\nabla \tilde{f}(q_0)|^2 + |\nabla H(z_0)|^2 \mathcal{O}(T e^{-T_0/T}) \\
&\hspace{3cm}+ \mathcal{O}(d_1 T^2) |\nabla \tilde{f}(q_0)||p_0|,
\end{align*}
provided that $T$ is chosen small enough such that $-p_1\T \Lambda p_1 T/2$ dominates the $\mathcal{O}(d_1 T^2) |p_1|^2$ terms and $-T^2 d_1 |\nabla \tilde{f}(q_0)|^2/4$ dominates the $\mathcal{O}(d_1 T^3) |\nabla \tilde{f}(q_0)|^2$ terms. Expanding the term $-T p_1\T \Lambda p_1/2$ and applying Young's inequality,
\begin{equation*}
C_\text{Y} d_1 T^2  p_0\T \nabla \tilde{f}(q_0) \leq \left( \frac{\sqrt{T}}{\sqrt{2}} \sqrt{d_1} |p_0| \right) \left( C_\text{Y} \sqrt{T}^3 \sqrt{2} \sqrt{d_1} |\nabla \tilde{f}(q_0)| \right) \leq \frac{T}{4} d_1 |p_0|^2 + C_\text{Y}^2 T^3 d_1 |\nabla \tilde{f}(q_0)|^2,
\end{equation*}
where $C_\text{Y}>0$ is an arbitrary constant, results in
\begin{align*}
V(q_2,p_2)-V(q_0,p_0) &\leq -\frac{T}{4} d_1 |p_0|^2 -\frac{T^2 d_1}{8} |\nabla \tilde{f}(q_0)|^2 + |\nabla H(z_0)|^2 \mathcal{O}(T e^{-T_0/T}).\\
&\leq -\frac{d_1 T^2}{8} |\nabla H(z_0)|^2 + |\nabla H(z_0)|^2 \mathcal{O}(T e^{-T_0/T}),
\end{align*}
where for the last inequality $T\leq 1$ has been used. Note that the term $T e^{-T_0/T}$ decays quickly for small $T$. Thus for $d_1 \geq \mathcal{O}(e^{-T_0/T}/T)$ the function $V$ strictly decreases along the trajectories of $\Phi_{d,T} \circ \Phi_T$. The lower bound on $d_1$ vanishes exponentially for small $T$. This concludes that the origin is asymptotically stable under the dynamics \eqref{eq:exDT} (with $\tilde{f}$ instead of $f$) provided that $T$ is chosen small enough, i.e. $T\leq T_\text{max}$, where up to the exponential terms in $d_1$, $T_\text{max}$ only depends on the ratio $d_2/d_1$, and $L_\text{H}$. Due to the fact that the first-order approximations of the dynamics resulting from $f$ and $\tilde{f}$ about the equilibrium are $\epsilon$-close (by construction of $\tilde{f}$), the origin is likewise asymptotically stable for the dynamics resulting from $f$, and its region of attraction contains at least a small neighborhood of the origin. Consequently, due to the continuity of $V$, we may choose $\epsilon$ small enough such that $V(\tilde{q}_2,\tilde{p}_2)-V(q_0,p_0)<0$ for all $(q_0,p_0)$ outside a neighborhood of the origin, where $\tilde{q}_2$, $\tilde{p}_2$ denotes the trajectory resulting from \eqref{eq:exDT} (with $f$ instead of $\tilde{f}$). This shows that $A$ is likewise contained in the region of attraction of the origin under the dynamics \eqref{eq:exDT} with $f$ instead of $\tilde{f}$.
\end{proof}


\section{Discrete-time stability (convex case)}\label{App:DisStabilityConvex}
By using the change of variables $(q_k,p_k)\rightarrow (\hat{q}_k,\hat{p}_k)$, 
\begin{equation}
\hat{q}_k=q_k+\left(\frac{\beta}{1-2dT}-T\right)p_k, \quad \hat{p}_k=p_k,
\end{equation}
the discrete-time algorithm (\ref{eq:exDT}) can be reformulated as
\begin{align}
\hat{q}_{k+1}&=y_k-T \tau \nabla f(y_k),\\
\hat{p}_{k+1}&=(1-2dT)\hat{p}_k-T\nabla f(y_k),
\end{align}
where $\tau:=\beta/(1-2dT)$, and 
\begin{align}
y_k&=q_k + \beta p_k\\
&=\hat{q}_k + T(1- 2d\tau) \hat{p}_k.
\end{align}
The change of variables is motivated by the fact that the convexity and smoothness of the objective function implies $f(\hat{q}_{k+1})\leq f(y_k)$,
\begin{equation}
f(\hat{q}_{k+1})\leq f(y_k) -|\nabla f(y_k)|^2 \left(T \tau- \frac{T^2 \tau^2 L}{2}\right),
\end{equation}
where $L$ is the Lipschitz constant of the gradient of $f$. In the following it will be shown that the energy $H$ decreases along $\hat{q}_k$ and $\hat{p}_k$ provided that certain conditions on the parameters $T,d,$ and $\beta$ are met. We evaluate $H(\hat{q}_{k+1},\hat{p}_{k+1})$,
\begin{align}
H(\hat{q}_{k+1},\hat{p}_{k+1})&=\frac{1}{2} (1-2dT)^2 |\hat{p}_k|^2 - T (1-2dT) \hat{p}_k\T \nabla f(y_k) + \frac{T^2}{2} |\nabla f(y_k)|^2 + f(\hat{q}_{k+1})\\
&\leq\frac{1}{2} |\hat{p}_k|^2  - (2dT-2d^2 T^2) |\hat{p}_k|^2 - T (1-2dT) \hat{p}_k\T \nabla f(y_k) \nonumber\\
&\hspace{3cm} -|\nabla f(y_k)|^2 \left(T \tau - \frac{T^2}{2} (\tau^2 L + 1)\right) + f(y_k).
\end{align}
Due to the fact that $f$ is convex it holds that 
\begin{equation}
f(y_k)-f(\hat{q}_k)\leq T(1-2d\tau) \nabla f(y_k)\T \hat{p}_k,
\end{equation}
and as a result
\begin{align}
H(\hat{q}_{k+1},\hat{p}_{k+1})-H(\hat{q}_k,\hat{p}_k)&\leq -(2dT-2d^2T^2) |\hat{p}_k|^2 +2dT (T-\tau) \hat{p}_k\T \nabla f(y_k) \nonumber\\
&\hspace{3cm}- |\nabla f(y_k)|^2 (T \tau - \frac{T^2}{2}(\tau^2 L+1)).
\end{align}
The right-hand side of the above expression is guaranteed to decrease provided that the following matrix 
\begin{equation}
\left( \begin{array}{cc} 2dT-2d^2 T^2 &  -dT (T-\tau) \\ -dT (T-\tau) & T \tau -  T^2 (\tau^2 L + 1)/2 \end{array} \right) \label{eq:DTStabCond}
\end{equation}
is positive definite. Thus, the minimum of $f$ is an asymptotically stable equilibrium of the dynamics \eqref{eq:exDT} provided that \eqref{eq:DTStabCond} is positive semidefinite. In particular, choosing $T=\tau$, i.e., $\beta=T(1-2dT)$, $0<T \leq 1/\sqrt{L}$, $0 < dT < 1$, ensures asymptotic stability of the minimum. Moreover, for any fixed $\beta>0$ the above matrix can be made negative definite by choosing $T$ and $d$ sufficiently small.


\section{Smoothness assumptions on $f$}\label{App:Smooth}
This section discusses the implications of Assumption~\ref{Ass:funfund} compared to Assumption~\ref{Ass:fun}. We therefore consider any function $f$ that satisfies Assumption~\ref{Ass:funfund} and construct a sequence of functions $f_j$, where $j>0$ is an integer, in the following way: Let $f_j:\mathbb{R}^n \rightarrow \mathbb{R}$ be such that $f_j(0)=0$, $\nabla f_j(0)=0$, and
\begin{equation}
\left.\frac{\partial^2 f_j}{\partial x^2}\right|_{x}:= \int_{\mathbb{R}^n} \left.\frac{\partial^2 f}{\partial x^2}\right|_{\bar{x}} s_j(x-\bar{x}) \diff \bar{x}, \label{eq:smooth}
\end{equation}
where $s_j:\mathbb{R}^{n} \rightarrow \mathbb{R}_{\geq 0}$ is an infinitely differentiable function that has support on $B_{1/(2j)}^{\text{c}}$ and satisfies $\int_{\mathbb{R}^n} s_j(\bar{x}) \diff\bar{x}=1$. As a consequence, it holds that
\begin{equation*}
|\nabla f(x) - \nabla f_j(x)|\leq 2\bar{C}_\text{f}/j, \quad |f(x) - f_j(x)|\leq 2\bar{C}_\text{f}|x|/j,
\end{equation*}
for all $x\in \mathbb{R}^{n}$ and $j > 0$ (assuming $\bar{C}_\text{f}>C_\text{f}$). In other words, $\nabla f_j$ converges uniformly to $\nabla f$; the same holds for $f_j \rightarrow f$ on any compact subset of $\mathbb{R}^n$. As a result of \eqref{eq:smooth}, any essential upper or lower bound on the curvature of $f$ are translated to the functions $f_j$. This implies, for example, that for large enough $j$ the origin is an isolated non-degenerate critical point of $f_j$. In addition, each $f_j$, $j>0$, is infinitely differentiable and satisfies Assumption~\ref{Ass:fun}. We therefore consider the dynamical system \eqref{eq:mombased2TV}, where $f$ is replaced with $f_j$:
\begin{align}
z_j^+(t)=g_\text{s}(t,z_j(t),\nabla f_j(q_j(t)), \quad \forall t\in I, t\geq t_0, \qquad z_j(t_0)=z_0, \label{eq:mombased2TVj}
\end{align}
where $z_j(t):=(q_j(t),p_j(t))$. Compared to \eqref{eq:mombased2TV}, the dependence on $\nabla f_j$ is made explicit. Let $z(t)$ be the trajectory satisfying \eqref{eq:mombased2TVj}, where $f_j$ is replaced with $f$. 

In the continuous-time case, \eqref{eq:mombased2TVj} implies
\begin{align}
|z_j(t)-z(t)|&\leq\int_{t_0}^{t} |g_\text{s}(\hat{t},z_j(\hat{t}),\nabla f_j(q_j(\hat{t}))) - g_\text{s}(\hat{t},z(\hat{t}),\nabla f(q(\hat{t})))| \diff \hat{t} \nonumber\\
&\leq C_\text{g1} \int_{t_0}^{t} |z_j(\hat{t})-z(\hat{t})| \diff \hat{t} + 2 C_\text{g2} (t-t_0) \bar{C}_\text{f}/j, \label{eq:boundsmooth}
\end{align}
where $C_\text{g1}>0$ and $C_\text{g2}>0$ are Lipschitz constants related to $g_\text{s}$. By virtue of the Gr\"{o}nwall inequality, \eqref{eq:boundsmooth} readily implies $z_j(t)\rightarrow z(t)$ pointwise for any $t\in I$. A similar argument applies to the discrete-time case. 

We show next that, in case the equilibrium at the origin is uniformly stable (uniformity with respect to $j$ and $t_0$), the convergence $z_j(t) \rightarrow z(t)$ is in fact uniform in $t\in I$. As a consequence, the results from Proposition~\ref{Prop:mainTV} and similarly Proposition~\ref{Prop:main} generalize to the case where $f$ satisfies merely Assumption~\ref{Ass:funfund} instead of Assumption~\ref{Ass:fun}, as shown below.

\begin{proposition}\label{Prop:unifConv}
Let the origin be a stable equilibrium for \eqref{eq:mombased2TVj}, uniformly in $j$ and $t_0$ (for $j>0$ sufficiently large). If $z(t)$ converges to the origin for $t\rightarrow \infty$, $z_j(t)$ converges to $z(t)$, uniformly in $t\in I$.
\end{proposition}
\begin{proof}
We pick any $\epsilon>0$ and show that there exists an integer $N$, independent of $t$, such that $|z_j(t)-z(t)|<\epsilon$ for all $j>N$ and all $t\geq t_0$. Due to the uniform stability of the origin, there exists a $\delta>0$ (independent of $t_0\in I$ and $j$) such that $|z_j(t_0)|<\delta$ implies $|z_j(t)|<\epsilon/2$ for all $t\geq t_0$ and $j>0$ sufficiently large. The trajectory $z(t)$ converges to the origin. Hence, there exists a finite time $T\in I$ such that $|z(T)|< \delta/2$. Applying the Gr\"{o}nwall inequality to \eqref{eq:boundsmooth} yields
\begin{equation}
|z_j(t)-z(t)|\leq 2 C_{\text{g2}} (T-t_0) \bar{C}_\text{f} e^{C_\text{g1} (T-t_0)}/j, \quad \forall t\in I, t_0\leq t\leq T;
\end{equation}
a similar conclusion holds in the discrete-time case. Thus, choosing $j$ large enough guarantees that $|z_j(t)-z(t)|< \delta/2\leq \epsilon/2$ for all $t\in I$, $t_0\leq t\leq T$. Therefore $|z_j(T)|< \delta$, which readily implies $|z_j(t)|< \epsilon/2$ for all $t\geq T$. Combined with $|z(t)|<\epsilon/2$ for all $t\geq T$, this yields $|z_j(t)-z(t)|<\epsilon$ for all $t\geq t_0$.
\end{proof}

Verifying that the origin is stable (uniformly in $j$) is typically straightforward. In order to illustrate the ideas we consider the discussion of Section~\ref{Sec:OptiEx1}. The total energy $H_j$ (where $f$ is replaced with $f_j$) is well-defined for any $j>0$. For large enough $j$, the curvature of $f_j$ is a localized average of the curvature of $f$. The origin is a non-degenerate isolated critical point of $f$, and therefore the same applies to $f_j$ for large enough $j$. This concludes that in a neighborhood of the origin the total energy $H_j$ is upper and lower bounded:
\begin{equation*}
    \frac{1}{2} |p|^2 + \frac{\mu}{4} |q|^2 \leq H_j(q,p) \leq \frac{1}{2} |p|^2 + L |q|^2,
\end{equation*}
for all $j>0$ sufficiently large. These upper and lower bounds are independent of $j$, which, combined with the fact that $H_j$ is non-increasing along trajectories, readily implies stability of the origin (uniformly in $j$ and $t_0$) \citep[see, for example,][p.~189]{Sastry}. The same argument applies in the discrete-time case, where, by assumption, $f$ and $f_j$ are locally strongly convex. Thus, uniform stability follows from Appendix~\ref{App:ProofDTStab}.

The fact that the convergence $z_j \rightarrow z$ is uniform in $t$ implies that the convergence estimates form Proposition~\ref{Prop:mainTV} and Proposition~\ref{Prop:main} apply for $j\rightarrow \infty$. More precisely, we have
\begin{proposition}
Let the origin be a stable equilibrium for \eqref{eq:mombased2TVi} uniformly in $j$ and $t_0$ (for $j>0$ sufficiently large). Let $A\subset \mathbb{R}^{2n}$ be compact and such that $z_0\in A$ implies $z(t)\rightarrow 0$ for $t\rightarrow \infty$. Then, provided that for sufficiently large $j>0$ the trajectories $z_j(t)$ satisfy the estimates
\begin{equation}
    |z_j(t)|\leq \hat{C}_j |z_0| \frac{\rho(t)}{\rho(t_0)} \exp(-\alpha (t-t_0)), \quad \forall t\in I, t\geq t_0, \quad \forall z_0\in A, \label{eq:tmpconv}
\end{equation}
where $\hat{C}_j>0$ and $\alpha>0$ are constant, and $\rho: \mathbb{R}_{\geq 0} \rightarrow \mathbb{R}_{>0}$ is continuous and monotonically decreasing, there exists a constant $\hat{C}$ such that $z(t)$ satisfies the estimate
\begin{equation}
    |z(t)|\leq \hat{C} |z_0| \frac{\rho(t)}{\rho(t_0)} \exp(-\bar{\alpha} (t-t_0)), \quad \forall t\in I, t\geq t_0, \quad \forall z_0\in A, \label{eq:tmpconv2}
\end{equation}
for any constant $\bar{\alpha}<\alpha$.
\end{proposition}
\begin{proof}
For any $\epsilon>0$, \eqref{eq:tmpconv} implies that for any $j>0$ sufficiently large,
\begin{equation}
    \lim_{t\rightarrow \infty} |z_j(t)|\frac{\rho(t_0)}{\rho(t)} \exp((\alpha-\epsilon) (t-t_0)) = 0.
\end{equation}
Due to the fact that $z_j$ converges uniformly to $z$, it follows that
\begin{equation}
    \lim_{t\rightarrow \infty} |z(t)|\frac{\rho(t_0)}{\rho(t)} \exp((\alpha-\epsilon) (t-t_0)) = \lim_{j\rightarrow \infty} \lim_{t\rightarrow \infty} |z_j(t)|\frac{\rho(t_0)}{\rho(t)} \exp((\alpha-\epsilon) (t-t_0)) = 0,
\end{equation}
which yields the desired result.
\end{proof}

\bibliography{literature}

\end{document}